\PassOptionsToPackage{authoryear,round}{natbib}

\documentclass[pdflatex,sn-mathphys-num]{sn-jnl}

\usepackage{graphicx}%
\usepackage{multirow}%
\usepackage{amsmath,amssymb,amsfonts}%
\usepackage{amsthm}%
\usepackage{mathrsfs}%
\usepackage[title]{appendix}%
\usepackage{xcolor}%
\usepackage{textcomp}%
\usepackage{manyfoot}%
\usepackage{booktabs}%
\usepackage{algorithm}%
\usepackage{algorithmicx}%
\usepackage{algpseudocode}%
\usepackage{listings}%
\usepackage{natbib}
\usepackage{enumerate}

\theoremstyle{thmstyleone}%
\newtheorem{theorem}{Theorem}
\newtheorem{proposition}[theorem]{Proposition}%

\theoremstyle{plain}%
\newtheorem{remark}{Remark}%
\newtheorem{notation}{Notation}%
\newtheorem{corollary}{Corollary}%
\newtheorem{lemma}[theorem]{Lemma}

\theoremstyle{thmstylethree}%
\newtheorem{definition}{Definition}%

\DeclareMathOperator*{\Argmax}{Argmax}

\usepackage{epstopdf}

\makeatletter
\renewcommand{\eqref}[1]{\textup{(\ref{#1})}}
\makeatother



\begin{document}

\title[Article Title]{A cartel-fringe model in a nonrenewable resource economy with many fringe firms extracting from a common deposit}

\renewcommand{\orcidlogo}{\includegraphics[width=.8em]{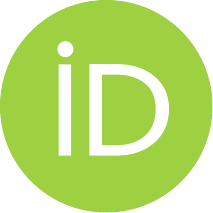}}

\author*[1]{\fnm{Agnieszka} \sur{Wiszniewska-Matyszkiel}}\email{agnese.mimuw@edu.pl}\equalcont{The authors contributed equally to this work.}

\author[2]{\fnm{Maciej} \sur{Wrona}}\email{ms.wrona@student.uw.edu.pl}
\equalcont{The authors contributed equally to this work.}

\affil[1]{\orgdiv{Institute of Applied Mathematics and Mechanics}, \orgname{University of Warsaw}, \orgaddress{\street{Banacha 2}, \postcode{02-097}, \city{Warsaw}, \country{Poland}}}

\affil[2]{\orgdiv{Faculty of Economic Sciences}, \orgname{University of Warsaw}, \orgaddress{\street{Długa 44/50}, \postcode{00-241}, \city{Warsaw}, \country{Poland}}}


\abstract{We study a model of a nonrenewable resource market, e.g. crude oil market. This market consists of a cartel with market power and a fringe consisting of many small firms, whose deposits are interrelated. In addition, the firms face constraints on extraction. Besides the nonrenewable resource, there is also its sustainable substitute, which constrains the price. We fully characterize the resulting Stackelberg equilibrium. Besides typical solution, in which initially the cartel and fringe extract simultaneously, we find that for some model parameters and initial capacities, the cartel may also deter the fringe from extraction, or it may refrain from extraction until the fringe depletes their deposit.
We conduct sensitivity analysis and study the conditions when one of those counterintuitive solutions is optimal.  }




\maketitle

\section{Introduction}\label{s_introduction}

Researchers modeling oil market most often assume that there are two agents in it: a cartel with monopolistic power and many small producers referred to as \emph{the fringe}. In such works, the fringe is both treated as one player---a single ``representative agent''---and it is a price taker: in their optimization, they treat the trajectory of oil prices as independent of their choices. In fact, those two assumptions together imply irrationality of the fringe. Moreover, this ``representative agent'' approach makes modeling the actual interactions within the group of small extractors impossible. In order to fully resolve the contradiction between rationality and price taking behavior, a continuum of fringe firms has to be considered instead. Moreover, it gives additional insight into the interaction within this group. One of such interactions, which we are going to focus on in our paper is when, despite the fringe firms having separate land and separate wells, all of them extract from a common deposit. Although, for the sake of brevity, we refer to the nonrenewable resource as \emph{oil} when discussing results and assumptions, our model is not restricted to oil, as similar interactions are inherent for, e.g.,  groundwater or water from deep deposits from the same aquifer. 

To be able to encompass all those aspects: lack of contradiction between and price-taking behavior of the fringe and its rationality, market power of the cartel and interactions among fringe firms, we model the market as a large game \citep[see e.g.][]{KALAI201859, SHITOVITZ1974458} with the cartel as an atom and the fringe as a continuum of small players. Such ''mixed markets'' constitute a convenient way to model imperfect competition.

The first to use a cartel-fringe model was \cite{salant1976exhaustible}, proposing a market model with one big cartel and competing fringe extractors. Salant has assumed that the oil producers' market is a ''dominant extractor'' market: it consists of a cartel, that is, a player proposing a price trajectory and production level in order to maximize profit, and a price-taking fringe, treated as one player.
He has proved that during the game, the cartel eventually becomes the sole entity in the oil market. In other words, it establishes a level of resource extraction such that the fringe players first deplete their resources. Moreover, he has justified the price-taking behavior of the fringe members results from the fact that it is the limiting behavior of a strongly asymmetric Cournot model. To do this, he assumed that constant fringe resources are divided equally among an increasing number of fringe members. In fact, the limit is equivalent to obtaining an atomic player and a continuum of nonatomic players. We will return to this concept in this paper.
Later, \cite{gilbert1978dominant} has proved that if the marginal cost of oil extraction by a collective of smaller extractors is constant, then the optimal cartel strategy is independent of the extraction costs of their resource and depends on the level of marginal cost of extraction by the fringe. He has also proved that the market price of oil steadily increases until the fringe depletes its resources. Along with \cite{newbery1981oil} and \cite{groot1992,groot2000open}, they have considered a case where the cartel is a leader (e.g. he has a first mover advantage) and fringe act as followers. Furthermore, \cite{Lewis1980} and \cite{loury1986} have also justified the assumption that the fringe is a price-taker. Furthermore, \cite{groot2000open} have obtained that the equilibrium price trajectory can be discontinuous.

From the game-theoretic point of view, cartel-versus-fringe models equilibria are Stackelberg equilibria, as the cartel maximizes payoff provided that the fringe best responds to its strategy, while each of fringe members just best responds to the cartel strategy and the strategies of the remaining fringe members. Given a strategy of the leader, the decision-making process of the fringe is a large game.  As mentioned before, a collective of identical fringe members is often modeled as one ``representative'' firm. This corresponds to the idea of \cite{HART1974453} that in a cooperative game such a behavior is a stable von Neumann-Morgenstern solution. Nevertheless, first, it requires that the fringe members cooperate and they all obey to the agreed strategy, optimal for their collective. Secondly, in such an approach, such a ``representative fringe firm'' is equivalent to another cartel, competing with the original cartel, as two atomic firms. They differ by size, but each of them has nonnegligible market power. If we consider a cartel-versus-fringe interaction with the fringe treated as one decision maker neglecting their influence on the price into account, then such a behavior may seem inconsistent. As the oil price has to equal the inverse demand, which also depends on the fringe decision, neglecting this influence in this case is irrational. This influence decreases as the abstract one entity called ``the fringe'' is decomposed into more and more actual small extractors, but as long as the number of fringe extractors is finite, some of this irrationality remains, which makes mathematical modeling problematic. To this end, we consider a large mixed market, as in \cite{SHITOVITZ1974458}, with the fringe modeled as a continuum  and the cartel as a single atomic player. This approach makes the fringe decisions consistent with reality.  Markets with a continuum of players, first introduced by \cite{aumann1964markets}, extended to games with continuum of players by e.g. \cite{schmeidler1973equilibrium} and then studied by e.g. \cite{KHAN199713}, have been introduced to resolve the aforementioned inconsistency, first for a competitive market. Later on, mixed large games have been introduced with both a continuum of small agents and some atomic players. Another extension of this primarily static concept are dynamic problems with large sets of players. Currently, they are sometimes dubbed mean-field games, analogously to mean-field models in physics, as, for instance, in the work of \cite{lasry2007mean} or \cite{chenavaz2021dynamic}. Properties of equilibria in dynamic games with a continuum of players that we need for our model were derived in \cite{wiszniewska2002static,wiszniewska2014open}.

The problem of cartel-versus-fringe models with representative fringe firms was thoroughly studied by Benchekroun and coauthors. They have solved their models by finding open-loop equilibrium strategies depending on the prevailing market phase: whether agents extract their resources simultaneously or only one of them sells at a specified price. \cite{BENCHEKROUN20091867} have given a full characterization of such a cartel-versus-fringe equilibrium. Furthermore, they have proved that at a market equilibrium, there is almost always a phase in which both players simultaneously extract their resources. \cite{BENCHEKROUN20191} have calculated the optimal control for both the cartel and the fringe in each phase. Consequently, they have determined the optimal depletion timing for the resources of both agents in the oil market. Besides, \cite{benchekroun2012price} derive the Nash equilibrium for the cartel-versus-fringe game with one cartel entity and the number of fringe extractors going to infinity and each of them having a separate deposit, both in open-loop and feedback strategies. A Stackelberg equilibrium in feedback strategies was also considered in \cite{GROOT2003287}.

The possibility that the miners extract a common pool deposit has already been considered in  \cite{feichtinger2024optimal}. Nevertheless, the authors have considered Nash equilibrium only, with $n$ symmetric firms. The common-pool situation has been compared with the situation of separate deposits, and the possibility of drilling new wells has been added.

For extensive reviews on resource economics, see, e.g., \cite{van2012applications,long2011dynamic,long2016resource}.
\medskip

In this paper, we consider a differential game of oil extraction closely related to \cite{BENCHEKROUN2020102384}. Typically, the model of \cite{BENCHEKROUN2020102384} considers
one cartel and a set of fringe members treated as one player. The control parameters of the fringe player in \cite{BENCHEKROUN2020102384} are unconstrained. Therefore, depending on cartel's choices, the supremum of the Hamiltonian function associated with
the fringe may be infinite, or the point at which the maximum is attained may be nonunique, or the fringe does not exploit. Thus, finding the optimal control of the fringe is problematic. \cite{BENCHEKROUN2020102384} solves this problem assuming that cartel chooses the control always leading to nonunique maximum of the fringe and choosing the fringe strategy such that the price is a continuous function. The equilibrium assumption is, consequently, that the price resulting from the fringe's choice is exactly the one that leads to the nonuniqueness. 

We modify the model of \cite{BENCHEKROUN2020102384} in three ways. Firstly, we impose constraints on the control parameters, which can be binding for the fringe. 
Secondly, we decompose the abstract fringe into many small companies. Thirdly, we assume that they exploit the same oil deposit using separate wells, as in one of the cases considered in \cite{feichtinger2024optimal}. Finally, we skip the assumption that price must be continuous, similarly to \cite{groot2000open}.

The structure of this paper is as follows. In Section \ref{s_formulation}, we present the model and formulate Stackelberg equilibrium. In section \ref{s_krytyka}, we discuss differences from the closest literature. 
In Section \ref{s_analysis}, we derive properties of the fringe reaction function to cartel's strategies, and calculate the optimal control for the cartel in several classes of strategies. Thus, we solve the optimization problem of the cartel by finding the maximum net profit in each class. After that, we present a graphical analysis within which we discuss how changing of various parameters influences the value of previously found discounted net profits. For transparency of exposition, more technical proofs and optimal control tools are in the Appendix.

\section{Formulation of the model}\label{s_formulation}

We consider a model closely related to \cite{BENCHEKROUN2020102384}. Similarly, we study a nonrenewable resource game between a coherent cartel and a group of fringe members. We start from open-loop strategies, although finally we conclude also about closed-loop or feedback strategies. Similarly, the cartel has monopolistic power, and it acts as a price maker, determining the market price, while fringe extractors are price takers, i.e., they treat the price as given. Similarly, the market presented in this study consists of two types of goods: non-renewable resources (e.g. oil) and its renewable perfect substitute. However, we obtain price-taking behavior of the fringe members as a result, not an assumption. We get it by treating each fringe member as an independent decision-maker. In our problem, each fringe extractor has their own well from which they extract oil, but all wells are drilled in the same oil deposit. The cartel extracts from a separate deposit. Due to a large number of fringe extractors, one should think of each of them as having infinitesimal influence on the joint resource level. 

As in \cite{BENCHEKROUN2020102384}, deposits or types of resource differ, and extracting from the deposit of the cartel is cheaper, with constant marginal costs. In case of oil, those two types may be ''conventional'' oil and  shale oil. The extraction cost is lower for the conventional oil, possessed by the cartel, and higher for the shale oil, extracted by the fringe. 
We use $k^f$ for the marginal cost of extraction from the fringe deposit and $k^c$ for the marginal extraction cost from mines belonging to the cartel. The \emph{initial stock of resources of the fringe} are $S_0^f$ and the \emph{initial stock of resources of the cartel} are $S_0^c$. The \emph{control parameters} for both fringe and cartel are \emph{extraction intensities} of the non-renewable resource: $q^f_i\in [0,m_i^f]$ for ${i \in F}$, where $F=[0,1]$ and $q^c\in \mathbb [0,M^c]$, respectively, where $m_i^f$, $M^c$ are intensity constraints of the firms, respectively. We assume that the function $i \mapsto m_i^f$ is measurable. We also assume that $M^c$ is large (such that a dumping strategy is possible and interior). This is the next essential difference from the original model, which had no constraints for extraction intensities. We consider open-loop controls $q^f_i:\mathbb R_+\to [0,m_i^f]$ for ${i \in F}$ and $q^c:\mathbb R_+\to \mathbb [0,M^c]$, respectively. We assume that strategy profiles of the fringe members $\left[q^f_i\right]_{i \in F}$,  are jointly measurable in $(i,t)$,  and $q^c$ is measurable. The resulting \emph{set of fringe strategy profiles} is denoted by $\mathbb S^f$ while \emph{set of cartel strategies} by $\mathbb S^c$. For technical reasons, we shall need the \emph{set of fringe decision profiles}, i.e., measurable functions $q^f: F\to \mathbb R_+$ with $q^f(i)\in [0,m_i^f]$. For simplicity, we write the player index $i$ in the subscript and write $q^f_i$.

We further denote the aggregate of a fringe strategy or decision profile by $Q^f$:
\begin{equation}\label{eq_Q^f_def}
    Q^f := \int\limits_F q_i^f di,
\end{equation}
where the integral is with respect to the Lebesgue measure. 

We denote the resulting set of aggregates of fringe decision profiles by $\overline {\mathbb D^f}$, while the set of aggregates of the fringe strategy profiles is denoted by $\overline {\mathbb S^f}$.

The controls
$    q_i^f$ (via their aggregate $Q^f$) and $q^c$ 
determine absolutely continuous trajectories of state variables $S^f:\mathbb{R}_+ \rightarrow \mathbb{R}$ for the fringe and $ S^c:\mathbb{R}_+ \rightarrow \mathbb{R}$ for the cartel. They are given by
\begin{equation}\label{zmiana_st_f}
    \dot{S}^f(t)=-Q^f(t) \text{ with } S^f(0)=S^f_0 
\end{equation}
and
\begin{equation}\label{zmiana_st_c}
    \dot{S}^c(t)=-q^c(t) \text{ with }S^c(0)=S^c_0.
\end{equation}

The state variables represent the amounts of respective resources left. 
The time when the fringe depletes its resources is denoted by $T^f$, while the time when the cartel depletes its resources is denoted by $T^c$.

Following \cite{BENCHEKROUN2020102384}, we assume that there is a potential perfect substitute for crude oil, which is renewable and produced at the constant\footnote{Compare with \cite{van2012there} for assuming a constant marginal cost of a backstop energy source.} marginal cost $b>0$. 
By slight abuse of notation, we assume that $b$ is the maximal price at which the production of the substitute is not switched on, which can be related to a very small but present set-up cost. We abstract from technological progress \citep[cf.][]{fischer2017balancing}. Let $p(t)$ be the price of crude oil at time $t$. A perfect renewable substitute for crude oil in the market imposes an upper bound on selling price, $b$, corresponding to the marginal cost of production of the perfectly substitutable good in place of oil. So, if $p(t)$ is set above $b$, then at time $t$ the demand for crude oil is equal to zero. Below, the inverse demand is linear, which we can write as 
\begin{equation}\label{eq_cena}
p(t)=\min\left\{\alpha-\beta\left(Q^f(t)+q^c(t)\right), b\right\},  
\end{equation}
where $\alpha,\beta>0$. 

We denote the discount rate by $r>0$. Moreover, we assume that $k^c<k^f<b<\alpha$ and that sharing the market with fringe extracting maximally always results in a price below $b$, which we obtain by $\alpha - \beta M^f<b$.


The payoff functions in the game are denoted by $J^c$ and $J^f_i$, respectively.

For fringe member $i$, their arguments are $(q^c,q_i^f,[q_j^f]_{j\in F\setminus\{i\}})$. 
So, the aim of each fringe member is to \emph{maximize} $$J_i^f(q^c,p,q_i^f,[q_j^f]_{j\in F\setminus\{i\}})=\int_0^{\infty}e^{-rt}(p(t)-k^f)q^f_idt,$$ which can be reduced to 
$J_i^f(q^c,q^f_i,Q^f)=\int_0^{\infty}e^{-rt}(p(t)-k^f)q^f_idt$\\
with $p(t)$ is given by \eqref{eq_cena}).

For the cartel, the arguments are $(q^c,Q^f)$, so the cartel \emph{maximizes}
$$J^c(q^c,Q^f)=\int_0^{\infty}e^{-rt}\left(\alpha-\beta(Q^f(t)+q^c(t))-k^c\right)q^c(t)dt.$$

Next, we define the main concept that we consider in our market model---Stackelberg equilibrium with multiple followers.

\begin{definition}\label{def_stackelberg}
\begin{enumerate}[(a)]
\item For arbitrary strategy of the leader $q^c\in \mathbb S^c$, define the \emph{reaction function of the followers} $\textup{RF}:\mathbb S^c\to \mathbb S^f$ as the set of  Nash equilibria in the game between fringe members (with $ q^c$ fixed), \\
i.e., the set of profiles of fringe strategies $\left[\bar q_i^f\right]_{i\in F}\in \mathbb S^f$ such that
\begin{equation}\label{eq_def_reaction}
\begin{aligned}
\textup{RF}(q^c)
:= \Bigl\{[\bar q_j^f]_{j\in F}:\,
& J^f_i\!\left(q^c,\bar q_i^f,[\bar q_j^f]_{j\in F\setminus\{i\}}\right)  \\
& \geq J^f_i\!\left(q^c, q_i^f,[\bar q_j^f]_{j\in F\setminus\{i\}}\right)  \\
& \text{for every } q_i^f\in [0,m_i^f], \text{ for a.e. } i\in F
\Bigr\}
\end{aligned}
\end{equation}

and the \emph{joint reaction function of the fringe} to $q^c$, $\textup{JRF}:\mathbb S^c\to \overline{\mathbb S^F}$ as the set of all aggregates of elements of those Nash equilibria, i.e.,
\begin{equation}
    \label{eq_def_joint_reaction}
    \textup{JRF}(q^c):=\left\{\int\limits_F q^f_i\ \textup{d}i: q^f\in \textup{RF}(q^c)\right\}.
\end{equation}
\item  A profile of strategies $\left(\bar q^c,\left[\bar q_i^f\right]_{i\in F}\right)\in \mathbb S^c\times \mathbb S^f$ is a Stackelberg equilibrium with player $c$ being a leader and multiple followers $\{i\in F\}$ if 
    \begin{enumerate}[(i)]
        \item $\left[\bar q_i^f\right]_{i\in F}$ is a Nash equilibrium in the fringe game given $q^c$, i.e.,\\
        $\left[q_i^f\right]_{i\in F}\in \textup{RF}(q^c)$, and
        \item $\bar q^c\in \Argmax\limits_{q^c\in \mathbb S^c}J^c\left(q^c,\textup{JRF}(q^c)\right)$.
    \end{enumerate}
\end{enumerate}
\end{definition}

Typically for Stackelberg problem, we start the description from the followers.

The optimization problem 
of fringe member $i$ remains unchanged: 
\begin{equation}\label{eq_fringe_payoff_stackelberg}
    J^f_i( q^c,q_i,[q_j^f]_{j\in F\setminus \{i\}})=\int_0^{\infty}e^{-rt}(p(t)-k^f)q^f_idt \rightarrow \text{max},
\end{equation}
with constraints $p(t)$ is given by \eqref{eq_cena}, 
$q^f_i(t) \in [0,m^f_i]$,
 $S^f$ is given by \eqref{zmiana_st_f} and 
$S^f(t)\geq 0$. 

As we can see, although $p$ appears as a constraint in such an optimization together with \emph{explicit} dependence on firm $i$ via \eqref{eq_Q^f_def}, this dependence is negligible. So, firm $i$ knows that the price is set such that the market clears and takes it into account together with its own actual influence.

The cartel's maximization problem at the Stackelberg equilibrium becomes
\begin{equation}\label{eq_kartel_calka_nieuproszczona}
    J^c(q^c,Q^f)=\int_0^{\infty}e^{-rt}p(t)q^c(t)dt \rightarrow \text{max},
\end{equation}
under the constraints $q^c(t) \in [0,M^c]$,
$p(t)$ given by \eqref{eq_cena}, $S^c$ given by \eqref{zmiana_st_c}, $Q^f\in \textup{JRF}(q^c)$, and $S^c(t)\geq 0$.

Note that decreasing $q^c(t)$ resulting in $\alpha-\beta(Q^f(t)+q^c(t))-k^c>b$ decreases both the current profit of the cartel and the state of its deposit, so it is never optimal. Thus, we can simplify the problem of the cartel to
\begin{equation}\label{eq_kartel_calka}
    J^c(q^c,Q^f)=\int_0^{\infty}e^{-rt}\left(\alpha-\beta(Q^f(t)+q^c(t))-k^c\right)q^c(t)dt \rightarrow \text{max},
\end{equation}
with constraints $q^c(t) \in [0,M^c]$, $b-\alpha+\beta(q^c+Q^f)\geq 0$,  $Q^f\in \textup{JRF}(q^c)$, 
and $S^c(t)$ given by \eqref{zmiana_st_c} fulfills 
$S^c(t)\geq 0$.


\section{Discussion on model assumptions and solution concepts versus previous literature}\label{s_krytyka}

In the work of \cite{BENCHEKROUN2020102384}, there is one representative fringe firm. The cartel maximizes its net profit and not only does it know the multivalued best response of the fringe to its extraction and price, but also the selection from it chosen by the fringe. Moreover, the chosen selection is the only strategy that guarantees that the market clears.  So, the fringe that neglects its influence on price in its optimization, chooses the strategy that takes into account its influence on market clearing.
Generally, Stackelberg equilibrium works well if the best response of the fringe is well defined and either is unique or, in the case of nonuniqueness, simple and economically justified rules of choice from this best response are needed.
Otherwise, although such a profile is a Stackelberg equilibrium, it is difficult to enforce the required fringe behavior---there is no obvious mechanism design.

 To this end, note that while analyzing the problem of the fringe, the authors have maximized the Hamiltonian in the form 
 \begin{equation}\label{ham_fringe}
    {H}^f(t,S^f(t),Q^f(t),\gamma^f(t))=e^{-rt}(p(t)-k^f)Q^f-\gamma^f(t)Q^f
 \end{equation}
  with $\gamma^f$ denoting the fringe costate variable\\
  over the unbounded set of strategies of the fringe. 
  
  Obviously, without an upper bound on the set of control parameters of the fringe, the supremum of the fringe's Hamiltonian given fixed price trajectory is either infinite or zero. 
 So, as a result of solving the problem in the form without additional assumptions imposed on the set of control parameters for the fringe, either:
\begin{itemize}
    \item there is no solution of the optimization of Hamiltonian (if in \eqref{ham_fringe}, $e^{-rt}(p(t)-k^f)-\gamma^f(t)>0$ for some $t$),
    \item or the solution is equal to zero
    (if $e^{-rt}(p(t)-k^f)-\gamma^f<0$ for some $t$),
    \item or every control leads to the same value of the Hamiltonian
(if $e^{-rt}(p(t)-k^f)-\gamma^f(t)=0$ for some $t$).
\end{itemize}
 The empty best response case is not a problem, as this never happens at equilibrium---the cartel will never restrict extraction that much. On the contrary, $e^{-rt}(p(t)-k^f)-\gamma^f(t)=0$ results in the best $Q^f(t)$ of the fringe completely arbitrary. In \cite{BENCHEKROUN2020102384}, the equilibrium strategy of the fringe was chosen as the unique one at which the market clears at the price $p(t)$ set by the cartel. This is consistent with one of the Stackelberg equilibrium definitions in wide sense for a nonunique best response. Nevertheless, the leader does not have tools to enforce such behavior. On the other hand, the fringe is assumed to neglect its influence on the price, so we cannot expect it to adjust its behavior in this way. Thus, we have no easily justifiable mechanism of the assumed rule of choice. Moreover, for the case of two players, the sense of negligibility means that an agent is not rational. Rational players can feel negligible only if there are many of them.

 As a result of maximization, both Hotelling rents, including $\gamma^f$, are constant. Having calculated $\gamma^f$, we can conclude that arbitrary $Q^f$ is in the set of best response of the fringe to such $p$. The authors have obtained an ambiguous result---after optimizing the fringe problem, the fringe strategy is chosen such that the commodity price remains at the level of $p(t)$, although the fringe does not take into account its influence on price and its best response is arbitrary. So, we have a Stackelberg-type interaction with a non-unique best response of the fringe (acting as the follower) with certain mechanism of choice from multiple best responses. 
 
 This was examined by e.g. \cite{caruso2020regularization} and \cite{carlier2019softening}, who used mathematical terminology of \emph{bilevel optimization} problem (in the language of economics or game theory: a Stackelberg equilibrium problem with single follower). They have presented different kinds of mathematical problems associated with non-uniqueness of the followers' reply and solution concepts to overcome the non-uniqueness and also approximation methods.
 
The mechanism of choice of how the fringe reacts in \cite{BENCHEKROUN2020102384} was not within the scope of those methods. Thus, in this paper, we present a different approach to the cartel-versus-fringe problem. We assume that the cartel chooses extraction and price trajectory at the beginning of the game such that the price has to be equal to market equilibrium price, taking into account the joint reaction of the fringe. In our model, due to explicit modeling of the interaction within the fringe, we obtain that the joint reaction is always unique and it is consistent with rationality of the fringe firms.

Despite the methodology of solving Stackelberg cartel-fringe models using representative fringe firm approach is quite well established, also in dynamic context, in each of them a certain amount of irrationality of the fringe remains.
 Introducing a continuum of fringe firms allows us to obtain rationality of all in coupled optimization problems with many decision makers and to drive results using formal mathematical analysis.


\section{Analysis}\label{s_analysis}
 
\subsection{The fringe game given a fixed strategy of the leader}

We start by showing that the fringe reaction to the cartel policy depends only on current extraction of the cartel, which we can write using an auxiliary \emph{static reaction function} and \emph{static joint best reaction function}. 

\begin{definition}\label{def_static_stackelberg}
\begin{enumerate}[(a)]
\item For an arbitrary decision of a leader $q^c\in [0,M^c]$, define the \emph{static reaction function of the followers} $\textup{SRF}(q^c)$ as the set of  Nash equilibria in the static game with fringe members only (with $ q^c$ fixed), \\
i.e., the set of profiles of fringe decisions $\left[\bar q_i^f\right]_{i\in F}$ with $q_i^f\in [0,m_i^f]$ such that

\begin{equation}\label{eq_def_static_reaction}
\begin{aligned}
\textup{SRF}(q^c)
:= \Bigl\{[\bar q_j^f]_{j\in F}:\,
& (\alpha - \beta (\bar Q^f + \bar q^c)) \bar q_i^f 
 \ge (\alpha - \beta (\bar Q^f + \bar q^c)) q_i^f  \\ 
& \text{for every } q_i^f \in [0, m_i^f], \text{ for a.e. } i \in F
\text{ if }S^f>0,\\
&\ q_i^f=0 \text{ otherwise}\Bigr\}
\end{aligned}
\end{equation}
and the \emph{joint static reaction function of the fringe} to cartel decision $q^c\in [0,M^c]$ as the set of all aggregates of elements of those Nash equilibria, i.e.,
\begin{equation}
    \label{eq_def_static_joint_reaction}
    \textup{JSRF}(q^c,S^f):=\left\{\int\limits_F q^f_i\ \textup{d}i: q^f\in \textup{SRF}(q^c,S^f)\right\}.
\end{equation}
\item By a slight abuse of notation, we extend the definition of \emph{joint static reaction function of the fringe} $\textup{JSRF}$ along entire cartel strategies, $q^c\in \mathbb S^c$, as the set of all $Q^f\in \overline{\mathbb S^f}$ with $Q^f(t)\in \text{JSRF}(q^c(t),S^f(t))$ for a.e. $t$, i.e.,
\begin{equation}
    \label{eq_def_static_joint_reaction_extended}
    \textup{JSRF}(q^c,S^f):=\left\{ Q^f\in \overline{\mathbb S^f}: Q^f(t)\in \textup{JSRF}(q^c(t),S^f(t))\text{ for a.e. }t\right\}.
\end{equation}
\item  A profile of strategies $\left[\bar q^c,\left[\bar q_i^f\right]_{i\in F}\right]$ is a \emph{myopic-follower Stackelberg equilibrium} with player $c$ being the leader and multiple followers $\{i\in F\}$ if 
    \begin{enumerate}[(i)]
        \item for a.e. t, $\left[\bar q_i^f\right]_{i\in F}(t)$ is a Nash equilibrium in the static fringe game given $q^c(t)$, i.e.,\\
        $\left[q_i^f\right]_{i\in F}(t)\in \textup{SRF}(q^c(t),S^f(t))$, and

        \item $\bar q^c\in \Argmax\limits_{q^c}J^c\left(q^c,\textup{JSRF}(q^c,S^f)\right)$.
    \end{enumerate}
\end{enumerate}
\end{definition}

Consequently, each Stackelberg equilibrium is a \emph{myopic-follower Stackelberg equilibrium}. First, there are some values of $q^c$ for which $\textup{SJRF}(q^c)$ is empty. Whenever it is nonempty, the values of $\textup{SJRF}(q^c)(t)$ are either 0 when it is not profitable for the fringe to extract whatever the other fringe firms do, or $M^f$ otherwise. 

\begin{proposition}
    \label{prop_JBR}
    \begin{enumerate}[(a)]
    \item Let $q^c$ be a strategy  of the cartel and let $\left[\bar q_i^f\right]_{i\in F}$ be a profile of strategies of the fringe. Then $\left[\bar q_i^f\right]_{i\in F}\in \textup{RF}(q^c)$ if and only if  for a.e. $t\in \mathbb R_+$, $q^f_{\cdot}(t)\in \textup{SRF}(q^c(t),S^f(t))$.
    \item A profile $\left(q^c,\left[\bar q_i^f\right]_{i\in F}\right)$ is a Stackelberg equilibrium in our game if and only if it is a myopic-follower Stackelberg equilibrium.
        \item  Let $q^c\in [0,M^c]$.   
The joint static reaction function of the fringe is uniquely defined by
\begin{equation}\label{eq_Q^f}\text{JSRF}(q^c,S^f)=
\begin{cases}
    \{0\}&\text{if } q^c\geq \frac {\alpha -k^f}{\beta}  \text{ or } S^f=0\\
    \{M^f\}&\text{if } q^c\leq \frac {\alpha - M^f-k^f}{\beta}\text{ and }S^f>0\\
    \emptyset & \text{otherwise}.
\end{cases}
\end{equation}
\item At every Stackelberg equilibrium for a.e. $t$ at which $S^f(t)>0$, either  $q^c(t)\geq \frac {\alpha -k^f}{\beta}$, $p(t)\leq k^f$ and $Q^f(t)=0$, or $\frac {\alpha - \beta M^f-b}{\beta}\leq q^c(t)\leq \frac {\alpha - \beta M^f-k^f}{\beta}$, $p(t)\geq k^f$ and $Q^f(t)=M^f$.
    \end{enumerate}
\end{proposition}

\begin{proof}
(a) and (b) We consider the notion of \textit{static equilibrium} at the time instance $t$ and state $x$. Following \cite{wiszniewska2014open}, it is a one-shot game played by continuum of players (such as the fringe) at the time instant $t$ and state $x$, in which players' set of strategies and the instantaneous payoff functions depend on both of time and state. By Theorem 5.1 of \cite{wiszniewska2014open} which states that every Nash equilibrium (in our case $q^f_i$ given $q^c$) such that for almost every player $i \in F=[0,1]$ the payoff is finite, then for all $t$, static profiles of players' strategies $q^f_.(t)$ are static equilibria at time $t$. Thus, every open-loop Nash equilibrium is a sequence  of static equilibria. Obviously, in our game, the equilibrium payoffs are nonnegative and bounded from above, thus finite for every player.
\\
(c) and (d) by basic calculus.
\end{proof}


\subsection{Extraction phases}\label{podrozdzial_fazy_cfe}

As the joint reaction of the fringe is not only discontinuous in $q^c$, but may also have empty values, the decision making problem of the cartel is far from standard problems for which Pontryagin Maximum Principle can be applied.

Therefore, following \cite{BENCHEKROUN2020102384}, we split the time interval into phases at which $q^c$ remains in one of subintervals of \eqref{eq_Q^f}. Assume that at the market equilibrium between the cartel and the fringe, there are four distinct phases of oil extraction (depending on who is currently extracting the resource in the market and what is the market price).

\begin{notation} By $F,C,S,L$ we denote the phases of resource extraction such that:
\begin{itemize}\label{ozn_fazy}
   \item  $F$ -- only the fringe firms extract,
   \item  $C$ -- only cartel extracts and sells at the price strictly lower than $b$,
   \item  $S$ -- both cartel and fringe extract,
   \item  $L$ -- only cartel extracts and sells at price $b$ (the upper bound for the price of the nonrenewable resource).  
\end{itemize}
The notation $t \in \Phi$ means that in the market at time $t$ phase $\Phi$ occurs, where $\Phi \in \{F,C,S,L\}$.
\end{notation}

Additionally, let us assume that throughout the entire game (more precisely, during the entire considered time horizon of market equilibrium) the phase can change. Consequently, sequences of phases are formed.

We expect that the market price of crude oil $p(t)$ at the moment of phase transition is discontinuous. Indeed, at majority of  equilibri, the price is going to jump up at least once. If the cartel sets price for which it is optimal for the fringe not to extract, then the jump appears at the moment at which the cartel depletes its deposit, while if it sets a higher price and extracts, as we shall prove, for $M^f$ high enough, the discontinuity will be when the fringe depletes its deposit. 

In economic reality, such discontinuity in the latter case implies that at least a part of the extraction should be stored and sold in the future. In this paper, we do not consider the availability of a sufficient storage technology. The storage may be omitted in the model for the following reasons: the fringe is myopic, it does not know when the market price will increase, or lacks substantial storage capacity due to high storage costs, resulting in a negative payoff under any market conditions. The model with introduced storage technology can be a source of further research.  Nevertheless, to reflect the real life conditions, storage capacities are not only bounded, but also small compared to extraction levels.

\subsection{Classes of cartel strategies}\label{podrozdzial_strategie}
We focus on finding the optimal strategy of the cartel given the reaction function of the fringe, considering the strategies it can adopt during the game. 

In the further part of the study, we have decided to consider the following three theoretical scenarios, i.e. classes of strategies that the cartel can adopt during the market game. Each of them consists of a single phase or a sequence of phases\footnote{We have introduced the phases $C,\ F,\ S,\ L$ in Subsection \ref{ozn_fazy}.}. Each phase depends on who is the active market participant (who extracts and sells oil) and the resulting selling price. Those classes of strategies can be characterized as follows:
\begin{itemize}\label{def_scenarios}

\item \textbf{1. ''Deter'' [the fringe from extraction]} The cartel extracts intensively, with $q^c\geq \frac{\alpha -k^f}{\beta}$, which decreases the price to at most $k^f$ until it depletes its reserves, which excludes the fringe from extraction ($Q^f(t)=0$). So, initially, phase $C$ occurs, and is final from the point of view of the cartel.

\item \textbf{2. ''Share'' [the market]} The cartel extracts at the same time as the fringe (phase $S$) and 

a) the cartel depletes its resources earlier than the fringe,

b) the cartel depletes its resources later than the fringe. Thus, at time $T^f$, which marks the depletion of fringe's resource, the cartel switches to the phase of oil extraction by sole agent---that is, phase $C$ or $L$.

\item \textbf{3. ''Wait'' [till the fringe depletes]}. The cartel waits until the fringe members deplete their resources, and after the depletion of fringe's resources at time $T^f$, phase $C$ or $L$ occurs. However, before the depletion of fringe's resource, the cartel does not extract the resource until the resources are depleted at time $T^f$, and after that phase $F$ occurs. Therefore, $q^c(t)=0$ for all $t<T^f$, which corresponds to phase $F$.
\end{itemize}

Note that the ''deter'' class of strategies is characterized by transitioning from phase $C$ to $F$, ''share'' strategy with cartel depleting first by transitioning from phase $S$ to $F$, ''share'' class of strategies with fringe depleting first by transitioning from phase $S$ to $C$ or $L$ and ''wait'' class of strategies by transitioning from phase $F$ to $C$ or $L$. Therefore, for all classes of strategies, there is a switch from one phase to another. In the case where we have a sequence of phases such that in each phase constituting the sequence, the cartel extracts oil, we calculate the optimal control of the cartel for each phase, starting from the last stage and proceeding backwards. Hence, in this study, we focus on analyzing phases $S$, $C$ and $L$ individually. This type of analysis is necessary because Pontryagin Maximum Principles (see:  Theorems \ref{ZMP_zwiazane} and \ref{zmp_zwykla} and Appendix \ref{appendix_fixed_state} and \ref{appendix_free_State}) require continuity of payoff in control, which is not satisfied when transitioning from one phase to another as $\textup{JSRF}$ is discontinuous.

The profit obtained from choosing a particular strategy is a combination of the profits from the initial phase and the final phase. We refer to the strategy for which the cartel's profit is the highest as the optimal strategy of the cartel.

\begin{definition}\label{startegia_optymalna}
   We refer to the optimal strategy of the cartel as the strategy that leads to achieving by the cartel the highest discounted net profit given by \eqref{eq_kartel_calka} for $Q^f=\textup{JSRF}(q^c,S^f)$ given by the formula \eqref{eq_def_static_joint_reaction} within one of the three strategy classes: ''deter'', ''share'' and '' wait''.
\end{definition}


In the sequel, first, we compute the optimal controls for the cartel in each considered phase, taking into account the adjusted model assumptions, posed in Chapter \ref{s_formulation}. Secondly, we determine the optimal extraction times for the cartel in each considered of ''deter'', ''share'' and ''wait'' classes of strategies and, for each class, we calculate the discounted net profits. Based on these results, we discuss the selection of the most advantageous strategy for the cartel.

\subsection{Candidates for cartel's optimal strategy and depletion time for both players}\label{obliczenie_ster_opt}
Similarly to \cite{BENCHEKROUN2020102384}, we solve the problem by the Pontryagin Maximum Principle. 

The Hamiltonian of the cartel can be initially written as $
   {H}^c(t,S^c, S^f,q^c,\gamma^c, \gamma^{cf})=e^{-rt}((\alpha -\beta (q^c+\textup{JSRF}(q^c,S^f)))-k^c)q^c-\gamma^c q^c - \gamma^{cf}\textup{JSRF}(q^c,S^f).$

\begin{remark}\label{rem_costate}
   \begin{enumerate}[a)]
       \item    The only dependence of $H^c$ on $S^f$ is via $\textup{JSRF}$, which is constant in each of the phases, and it does not influence the optimizer of $H^c$ within the phase. So, it is enough to consider the Hamiltonian in a simpler form
    \begin{equation}\label{eq_hamiltonian_cartel}
   {H}^c(t,S^c, S^f,q^c,\gamma^c)=e^{-rt}((\alpha -\beta (q^c+\textup{JSRF}(q^c,S^f)))-k^c)q^c-\gamma^c q^c.   
\end{equation}

\item At the Stackelberg equilibrium, $H^c$ is always well defined, while $q^c(t)$ for $t$ for which $S^f(t)>0$, is restricted to two disjoint intervals $\left[\frac{\alpha-k^f}{\beta},\frac{\alpha}{\beta}\right]$, $\left[\frac{\alpha-\beta M^f-b}{\beta},\frac{\alpha-\beta M^f-k^f}{\beta}\right]$. So, calculating a Stackelberg equilibrium requires considering a smaller, disjoint set of controls.

\item As $H^c$ is independent of $S^c$, $\gamma^c$ is constant in each of the phases.
    \end{enumerate}
\end{remark}

\subsubsection{The final phase: the cartel remains the sole producer, extracting resources after the fringe depleted their resources}\label{faza_koncowa_1}
This phase occurs after the phase of joint extraction of resources by the cartel and the fringe firms, or after the depletion of resources by the fringe, who was the sole player in the market due to waiting by the cartel. We show that if fringe firms deplete the resource and the cartel becomes the sole player in the market, then the cartel extracts its resource at the market price $b$. Thus, in accordance with the Notation \ref{ozn_fazy}, the last phase will be 
$L$. For the purposes of the proof, we consider the boundary values of the interval that restrict the set of cartel control parameters.

Now, we check whether the optimal control of the cartel is equal to the upper or lower bound of the interval $[\frac{\alpha-b}{\beta},\frac{\alpha}{\beta}]$.

\begin{proposition}\label{stw_o_brzegowym_ster_opt}
    There exists $\overline{t}<T^c$ such that the optimal control of the cartel in the discussed phase is equal to the lower bound of the interval $q^c=\frac{\alpha-b}{\beta}$ for a.e. $t\in [\overline{t},T^c]$, and $q^c>\frac{\alpha-b}{\beta}$ for a.e. $t$ in this phase before $\overline{t}$.
\end{proposition}
\begin{proof}
  See Appendix \ref{ss_proofs_stw_o_brzegowym_ster_opt}.
\end{proof}

Based on condition \eqref{max_zero} of Theorem \ref{ZMP_zwiazane}, we have $\max_{q^c}H^c\big(T^c,S^c_*(T^c),q^c,\gamma^c(T^c)\big)=0$. Substituting the form of the optimal control at the final time $q_{\max}^c(T^c)=\frac{\alpha-b}{\beta}$ into the Hamiltonian function, we obtain
\begin{equation*}
    \frac{\alpha-b}{\beta}\Big(\big(\alpha-k^c-\beta \cdot \frac{\alpha-b}{\beta}\big)e^{-rT^c}-\gamma^c\Big)=0,
\end{equation*}
thus $(b-k^c)e^{-rT^c}-\gamma^c(T^c)=0$, which implies $\gamma^c(T^c)=(b-k^c)e^{-rT^c}$. Consequently, $\gamma^c(T^c)=\gamma^c>0$, since it is given that $b>k^c$.

Now, we impose appropriate conditions on the parameters in the model so that the cartel's control function is \textit{everywhere boundary}, meaning for every $t$ such that $T^f\leq t \leq T^c$, which is for every $t$ in the discussed phase, the equality $q^c_*(t)=\frac{\alpha-b}{\beta}$ holds almost everywhere.

\begin{proposition}\label{stw_o_ograniczeniu_wszedzie_brzegowa} 
If $2b-k^c\leq \alpha$, then for every $t$ such that $T^f\leq t \leq T^c$ in the discussed phase, the optimal control $q^c_*$ is equal to $\frac{\alpha-b}{\beta}$, the price is $b$ and the depletion time is $\hat{T}^c=T^f+\frac{S_0^c \cdot \beta}{\alpha-b}$.
\end{proposition}

\begin{proof}
See Appendix \ref{ss_proofs_stw_o_ogr_wszedzie_brzegowa}.
 \end{proof}

The assumption $2b-k^c\leq \alpha$ makes economic sense, because $\alpha$ represents the lowest price at which consumer demand drops to zero, i.e. is a choke price (there is no demand for the good at such a low price). On the other hand, $b$ is the upper bound on prices set by the monopolist to prevent entry of competitors producing perfect substitutes for crude oil, i.e. is a limit price (it is a maximum price that the cartel is willing to pay for oil). For further analysis, we can additionally assume that $\alpha$ is significantly higher than the other price parameters, namely the marginal costs $k^c$ and $k^f$.




\subsubsection{The final phase from the cartel's perspective: the cartel extracts as a sole extractor at a price $p(t)\leq k^f$}

First, we note that decreasing price by increasing extraction below $\frac{\alpha-k^f}{\beta}$ is suboptimal, which we can prove analogously to the case of extraction by the cartel after the fringe depleted.

This is the final phase from the cartel's perspective, denoted as $C$ according to Notation \ref{ozn_fazy}. By Proposition \ref{prop_JBR}, $Q^f(t)=0$ for this extraction phase. 
\begin{proposition}\label{prop_k_f_boundary}
    If 2$k_f-k^c\leq \alpha$, the for a.e. $t$ in this phase $q^c=\frac{\alpha-k^f}{\beta}$, the price is $k^f$ and the cartel's depletion time is $\tilde{T}^c=\frac{S_0^c \cdot \beta}{\alpha-k^f}$.
\end{proposition}
\begin{proof}
    Analogously to proof of Proposition \ref{stw_o_ograniczeniu_wszedzie_brzegowa} with $b$ replaced by $k^f$.
\end{proof}

It is straightforward to derive that since $p(t)=\alpha-\beta q^c$ during this phase and $p(t)=k^f$, the cartel's control under this constraint is $q^c_*(t)=\frac{\alpha-k^f}{\beta}$. Now, we calculate the extraction time for the cartel resulting from this control (denoted as $\tilde{T}^c$). We need to compute the integral:
\begin{equation*}
    \int_0^{\tilde{T}^c}\frac{\alpha-k^f}{\beta}dt=S^c_0.
\end{equation*}
Hence, $\frac{\alpha-k^f}{\beta}t\Big|_0^{\tilde{T}^c}=S_0^c$ and after simple transformations, we obtain $\tilde{T}^c=\frac{S_0^c \cdot \beta}{\alpha-k^f}$.

\subsubsection{The initial phase: the cartel extracts resources at the same time as the competitors}\label{ref_faza_poczatkowa}

In accordance with Notation \ref{ozn_fazy}, we consider the phase of simultaneous oil supply denoted as $S$. In this subsection, we calculate the value of the constant $\gamma^c$ for this phase, demonstrate the existence of the optimal extraction time for the cartel, and then compute the optimal control level $q^c(t)$ for the cartel in this phase. 

Based on the information that the cartel and the fringe extract oil simultaneously in this phase, we conclude that the cartel sets the price within the range $p(t) \in (k^f,b)$, where $k^f$ is the marginal cost of extraction for the fringe and $b$ is the limit selling price of crude oil. According to the definition of this phase, we assume that the price is greater than $k^f$, as the cartel sets a price at which the competitors will make profits upon entering the market. Hence, according to equation \eqref{eq_Q^f}, they will extract resources at intensity $Q^f=M^f$. 

We assume that the optimal control for the cartel in this phase does not lead to price $b$.
To exclude the possibility of obtaining the price $b$ as optimal in this phase, we impose an additional constraint on the model parameters later in this subsection.

In the discussed phase, the price is set within the interval $(k^f,b)$, which means that the set of optimal control parameters for the cartel is bounded on both sides. We determine the lower bound of this interval by considering the upper limit imposed on the price (assuming that the price in this phase is less than $b$):
\begin{equation*}
    p(t)=\alpha-\beta(q^c+M^f)< b,
\end{equation*}
where $Q^f=M^f$ represents the optimal level of resource extraction by the competitors. From this it follows that $q^c(t)> \frac{\alpha-b-M^f\beta}{\beta}$.

The upper bound is determined similarly from the lower limit on the price:
\begin{equation*}
   p(t)= \alpha-\beta(q^c+M^f) > k^f,
\end{equation*}
which implies $q^c(t)< \frac{\alpha-k^f-M^f\beta}{\beta}$. Hence, the optimal control for the cartel belongs to the interval:
\begin{equation}\label{scen3_przedzial_par_ster}
    q^c(t) \in \Big(\frac{\alpha-b-M^f\beta}{\beta},\frac{\alpha-k^f-M^f\beta}{\beta}\Big).
\end{equation}
We assume that $\frac{\alpha-b-M^f\beta}{\beta}>0$. Note that in Subsection \ref{faza_koncowa_1} we have assumed that $\frac{\alpha-b}{\beta}>0$, so $\frac{\alpha-b}{\beta}>M^f$. Hence, we strengthen the previous assumption.

In proving the existence of the optimal exhaustion time $T^c$ and while calculating $\gamma^c$, we rely on the assumption that the cartel will deplete its resources later than the competitors, provided that it allows them to exploit theirs. To ensure this, we  impose appropriate constraints on the resources of the cartel and the fringe: $S_0^c$, $S_0^f$. This assumption is reasonable because we assume that the reserves of the competitors are much smaller than those of the cartel.

\begin{proposition}\label{stw_o_ograniczeniu_zasobow}
If $2b-k^c\leq \alpha$ and $\frac{S_0^f}{M^f} < \frac{S_0^c \beta}{\alpha-\beta M^f - k^f}$, then the cartel will not deplete its resources before the fringe at the simultaneous extraction phase.
\end{proposition}
\begin{proof}
See Appendix \ref{ss_proofs_stw_o_ograniczeniu_zasobow}.
\end{proof}

Now, we calculate the constant $\gamma^c$, which is useful for proving the existence of the optimal depletion time for the cartel's resources. First, let us introduce the definition of the function $\pi^c_L$ and impose additional constraints on the model parameters.

Since we consider the phase of simultaneous supply $S$, the depletion time of the resource of the fringe $T^f$ is fixed and smaller than $T^c$ and there are no switches between phases. For the purpose of optimization, we can use the maximum principle with fixed time and free final state, i.e., Theorem \ref{zmp_zwykla}. Using it, we obtain a terminal condition for $\gamma^c$:
\begin{equation}\label{war_kon_gamma}
  \gamma^c(T^f)=\frac{\partial \pi^c_L\big(T^f,S^c_*(T^f)\big)}{\partial S^c}e^{-rT^f},    
\end{equation}
where $\gamma^c(T^f)$ is the costate variable of the cartel, satisfying the terminal condition \eqref{zmp_2_wk} in Theorem \ref{zmp_zwykla}.

We define the function $\pi^c$ as the cartel's terminal payoff function, which depends on the depletion time of the competitors' resources, denoted as the end of the discussed phase:
 \begin{equation}\label{wypl_konc}
     \pi^c_L\big(T^f,S^c_{T^f}\big):=\int_{T^f}^{T^c}e^{-rt}\big(\alpha-k^c-\beta q^c_*(t)\big)q^c_*(t) dt,
 \end{equation}
 where $q^c_*(t)$ is the optimal control in phase $L$, according to Proposition \ref{stw_o_ograniczeniu_wszedzie_brzegowa} equal to $\frac{\alpha-b}{\beta}$. Here, $S^c_{T^f}:=S^c(T^f)$ represents the state of the cartel's resources at time $T^f$. We changed the notation to emphasize that the function $\pi^c_L$ is a function of $\pi^c_L(t,S^c)$ (where we substitute $t$ with $T^f$ in the next part of the work).

The equation in which the optimal depletion of the cartel's resource appears is obtained by solving the integral:
\begin{equation}\label{caleczka}
S_0^c=\int_0^{T^c}q^c(t) dt=\int_0^{T^f}q^c(t)dt+\int_{T^f}^{T^c}\frac{\alpha-b}{\beta}dt.
\end{equation}
The optimal control of the cartel over the interval $t \in [0,T^f]$, which is necessary to establish the existence of the moment $T^c$, is found by maximizing the Hamiltonian of the cartel problem as stated in Theorem \ref{zmp_zwykla} (we use Hamiltionian of the form given by \ref{ham_fringe}). For now, we assume that the set of optimal control parameters is within the interior of the interval given in \eqref{scen3_przedzial_par_ster}. The necessary condition \eqref{argmax_u_zwykla} of Theorem \ref{zmp_zwykla} for points $q^c$ from the interior of the interval \eqref{scen3_przedzial_par_ster} looks as follows:

\begin{equation}\label{war_kon_trojka}
e^{-rt}(\alpha-k^c-M^f\beta-2\beta q^c)-\gamma^c(T^f)=0.
\end{equation}
After simple transformations, we obtain:
\begin{equation}\label{opt_ster_trojka}
q^c_*(t)=\frac{\alpha-k^c-M^f\beta}{2\beta}-\frac{\gamma^c}{2\beta}e^{rt}.
\end{equation}
Substituting the optimal control into \eqref{caleczka}, we get:
\begin{equation*}
S_0^c=\int_0^{T^c}q^c(t)dt=
\int_0^{T^f}\frac{\alpha-k^c-M^f\beta}{2\beta}-\frac{\gamma^c}{2\beta}e^{rt}dt+
\int_{T^f}^{T^c}\frac{\alpha-b}{\beta}dt.
\end{equation*}
By calculating the integrals of functions of the form $\text{constant}_1\cdot e^{\text{constant}_2\cdot t}$ and constant functions, we obtain:
\begin{equation}\label{S^c_0_strategia_druga}
\frac{\alpha-k^c-M^f\beta}{2\beta}T^f-\frac{\gamma^c(T^f)}{2\beta r}e^{rT^f}+\frac{\gamma^c(T^f)}{2\beta r}+\frac{\alpha-b}{\beta}(T^c-T^f)=S_0^c.
\end{equation}

Notice that $S^c(T^f)$, the remaining cartel's resource at time $T^f$, is equal to the last term in equation \eqref{S^c_0_strategia_druga}, i.e. $S^c(T^f)=S^c_{T^f}=\frac{\alpha-b}{\beta}(T^c-T^f)$. Thus $T^c=T^f+\frac{S^c_{T^f}\cdot\beta}{\alpha-b}$.
Substituting this derived $T^c$ into equation \eqref{wypl_konc}, we obtain:
\begin{equation*}
   \pi^c\big(T^f,S^c_{T^f}\big)=\int_{T^f}^{T^c}e^{-rt}\Big(\alpha-k^c-\beta \cdot \frac{\alpha-b}{\beta}\Big)\cdot\frac{\alpha-b}{\beta} dt.
\end{equation*}
After integrating a constant function and substituting $T^c$ and $T^f$ as the integration limits, we get:
\begin{equation}\label{pi^c}
    \pi^c\big(T^f,S^c_{T^f}\big)=\frac{1}{r}(b-k^c)\cdot\frac{\alpha-b}{\beta}\cdot\big(e^{-rT^f}-e^{-rT^c}\big).
\end{equation}

\begin{proposition}\label{stw_o_gamma^c}
If $2b-k^c\leq \alpha$ and $\frac{S_0^f}{M^f} < \frac{S_0^c \beta}{\alpha-\beta M^f - k^f}$, then the constant $\gamma^c$ fulfils $\gamma^c=e^{-rT^f-r\frac{S^f(T^f)\beta
}{\alpha -b}}(b-k^c)=(b-k^c)e^{-r(T^f+T^c)}$.
\end{proposition}
\begin{proof}
    See Appendix \ref{ss_proofs_stw_o_gamma_c}.
\end{proof}

Now we want to ensure that $q^c_*(t)$ given by \eqref{opt_ster_trojka} is greater than $\frac{\alpha-b-M^f\beta}{\beta}$ for every $t<T^f$. We aim to satisfy condition \eqref{argmax_u_zwykla} from Theorem \ref{zmp_zwykla}---this condition is only met for $q^c$ inside the interval specified by \eqref{scen3_przedzial_par_ster}. At this point, claiming that the obtained optimal control would imply a price lower than $b$ is not a consequence of any previous assumptions in our model; thus, it is yet unjustified. Additional constraints on the parameters of the considered model must be imposed. 

\begin{proposition}\label{stw_o_nierownosci_z_dyskontem}
    If  $2b-k^c\leq \alpha$, $b(2-e^{-rT^c})>\alpha-M^f\beta+k^c(1-e^{-rT^c})$ and  $\frac{S_0^f}{M^f} < \frac{S_0^c \beta}{\alpha-\beta M^f - k^f}$, then the optimal control of the cartel $q^c_*(t)$ for this phase is greater than $\frac{\alpha-\beta M^f-b}{\beta}$  and $p(t)<b$.
\end{proposition}
\begin{proof}
    See Appendix \ref{ss_proofs_stw_o_nierownosci_z_dyskontem}.
\end{proof}

Now, we show that the optimal depletion time $T^c$ exists. Notice that by rearranging equation \eqref{S^c_0_strategia_druga} we obtain:
\begin{equation}\label{tece}
T^c=\frac{(S_0^c+s)\beta}{\alpha-b},
\end{equation}
where $s$ is a positive constant, independent of $S_0^c$. Using the formula for $\gamma^c$ from equation \eqref{gamma_c}, we can express $s$ explicitly as follows:
\begin{equation}\label{es}
s=\frac{\alpha+k^c+M^f\beta-2b}{2\beta}\cdot T^f+\frac{(b-k^c)}{2\beta r}\cdot\Big(e^{-rT^f}-1\Big)e^{-r(T^c+T^f)}, \text{ where } T^f=\frac{S_0^f}{M^f}.
\end{equation}

\begin{proposition}\label{stw_o_skonczonsci_optymalnego_momentu_wyczerpania}
    The optimal depletion time $T^c$ for the cartel's resources exists and it is finite.
\end{proposition}

\begin{proof}
    See Appendix \ref{ss_proofs_stw_o_skonczonosci_optymalnego_momentu}.
\end{proof}
    
Finally, we calculate the explicit form of optimal control in the considered phase $S$.

\begin{proposition}
    If  $2b-k^c\leq \alpha$, $b(2-e^{-rT^c})>\alpha-M^f\beta+k^c(1-e^{-rT^c})$ and  $\frac{S_0^f}{M^f} < \frac{S_0^c \beta}{\alpha-\beta M^f - k^f}$, then the optimal cartel control at phase $S$ is given by $q^c_*(t)=\frac{\alpha-k^c-M^f\beta}{2\beta}-\frac{(b-k^c)}{2\beta}\cdot e^{r(t-T^f-T^c)}$. If the last inequality is not fulfilled, then $q^c_*(t)=\max\left\{\frac{\alpha-b-M^f\beta}{\beta} ,\frac{\alpha-k^c-M^f\beta}{2\beta}-\frac{(b-k^c)}{2\beta}\cdot e^{r(t-T^f-T^c)}\right\}$.
\end{proposition}
\begin{proof}

    The calculation of the optimal control follows from the necessary condition of Theorem \ref{zmp_zwykla}, i.e. equation \eqref{argmax_u_zwykla}. Based on this condition, we obtain the form of the optimal control \eqref{opt_ster_trojka}. Substituting the obtained expression for $\gamma^c(T^f)=(b-k^c)e^{-r(T^f+T^c)}$ into the cartel's control function, we ultimately get the explicit form $q^c_*(t)=\frac{\alpha-k^c-M^f\beta}{2\beta}-\frac{(b-k^c)}{2\beta}\cdot e^{r(t-T^f-T^c)}$, with $T^c$ being the solution to \eqref{rownanie_z_tece}.
\end{proof}

\subsubsection{Various information structures}

Equilibria in open-loop strategies of the players are often criticized for not being strongly subgame perfect and being unrealistic. Equilibria in feedback strategies, i.e., strategies dependent on the current state, are strongly subgame perfect and are regarded as more realistic. Therefore, an obvious question is what if we change the form of strategies of the players to \emph{feedback strategies}, i.e., we consider strategies of the form $q^{c,F}(S^c,S^f,t)$ and $q_i^{f,F}(S^c,S^f,t)$ and we want the equilibrium condition to hold for every initial condition, not only the fixed one. 
We are interested in how this change influences the equilibrium extractions and trajectories of the resources. 
For this comparison, we need the \emph{open-loop forms} of feedback strategies, i.e., open-loop strategies describing the choices at various time instants. They are defined as follows.

\begin{equation}
    \label{eq_feedback_OL_forms} 
\begin{aligned}
\textup{OL}\Bigl(q^{c,F}\bigl(S^c,S^f,t\bigr)\Bigr)(t) 
&= q^{c,F}\bigl(S^c(t),S^f(t),t\bigr), \\[2mm]
\textup{OL}\Bigl(q_i^{f, F}\bigl(S^c,S^f,t\bigr)\Bigr)(t) 
&= q_i^{f,F}\bigl(S^c(t),S^f(t),t\bigr).
\end{aligned}
\end{equation}

\begin{proposition}
    The equilibrium in feedback strategies is equivalent to the equilibrium calculated for open-loop strategies, i.e., for every initial condition $\left(S_0^c,S_0^f\right)$,  the open-loop form of feedback equilibrium strategies coincide almost everywhere with open-loop equilibrium strategies.
\end{proposition}
\begin{proof}
Consider a feedback Stackelberg equilibrium $\left(\bar q^{c,F},  \left[ \bar q_i^{f,F}\right]_{i\in F}\right)$ and an initial condition $\left(S_0^c,S_0^f\right)$  leading to the trajectory of the state $\left(\bar S^c,\bar S^f\right)$. Consider the game played by the fringe firms given $\bar q^{f,F}$. It is a differential game with a set of players $[0,1]$ and the decision space $\mathbb R_+$ and all the functions are continuous, so, assumptions of Theorem 5.1(d) of \cite{wiszniewska2002static} are satisfied. 
Thus, for a.e. $t$, each feedback Nash equilibrium is composed of static Nash equilibria at time $t$ and state $\left(\bar S^c(t), \bar S^f(t)\right)$. Let the symbol OL of a profile denote its open-loop form, i.e., strategies of all players rewritten as a function of time. The best response of a fringe firm $i$ is $m_i^f$ if $\textup{OL}\left(q^{c,F}\right)(t)+\textup{OL}\left(Q^{f,F}\right)(t)$ results in $p(t)>k^f$ , is unique and equal to $m_i^f$, and this holds for every $i$. Otherwise, if $p(t)<k^f$, then firm $i$ will not produce, and this holds for every $i$. If $p(t)=k^f$, then firm $i$ is indifferent between arbitrary levels of production as the profit is 0. So, the joint reaction of the fringe with tie-breaking rule is given by \eqref{eq_Q^f}. At the aggregate level, ``it is not profitable to produce'' reduces to $\textup{OL}\left(q^{c,F}\right)(t)\geq \frac{\alpha -k^f}{\beta}.$ ``It is profitable to produce'' implies that $\textup{OL}\left(q^{c,F}\right)(t)<\frac{\alpha -k^f-\beta M^f}{\beta}$ and there are no equilibria with $\textup{OL}\left(q^{c,F}\right)(t)$ in between. So, the behavior of the fringe depends on $q^{c,F}$ only by the open-loop form of $\textup{OL}\left(q^{c,F}\right)(t)$.
After incorporating the reaction of the fringe, the leader's problem becomes purely an optimal control problem. Thus, the optimal control in the class of open-loop controls, is the open-loop form of the optimal feedback control.
\end{proof}

\subsection{Comparison of cartel's net profit between strategies}

In this subsection, we calculate the discounted net profits for ''deter'', ''share'' for fringe depleting first \footnote{We do not calculate the profit for the other type of ''share'' strategy, as in Section \ref{s_formulation}, we assumed that the cartel will deplete its resources later than the fringe firms.} and ''wait''. Then, we compare and discuss them to determine what information about the model parameters the cartel should possess to indicate the optimal strategy at the beginning of the game.

\subsubsection{Optimal profit over the class of ''Deter'' strategies (see page \pageref{def_scenarios})}\label{s1}

We first calculate the profit for the optimal ''Deter'' strategy of the cartel. For this strategy, finding the result involves calculating the integral:
\begin{equation*}
\pi^c_1:=\int_{0}^{\tilde{T}^c}e^{-rt}\big(\alpha-\beta \cdot \frac{\alpha-k^f}{\beta}-k^c\big)\cdot\frac{\alpha-k^f}{\beta} dt.
\end{equation*}
After evaluating the integral of the elementary function, we obtain:
\begin{equation*}
-\frac{1}{r}e^{-rt}(k^f-k^c)\cdot\frac{\alpha-k^f}{\beta} \Big |^{\tilde{T}^c}_0=\frac{1}{r}(k^f-k^c)\cdot\frac{\alpha-k^f}{\beta}\cdot(1-e^{-r\tilde{T}^c}),
\end{equation*}
where $\tilde{T}^c= \frac{S_0^c \cdot \beta}{\alpha - k^f}$. Therefore, the net profit in the first case is given by:
\begin{equation}\label{zysk1}
\pi^c_1=\frac{1}{r}(k^f-k^c)\cdot\frac{\alpha-k^f}{\beta}\cdot(1-e^{-r\frac{S_0^c \cdot \beta}{\alpha - k^f}}).
\end{equation}

\subsubsection{Profit for optimal over ''Share'' strategies of the cartel (see page \pageref{def_scenarios})}\label{s2}

The discounted net profit for the cartel under ''share'' strategy with fringe depleting first is calculated based on the integral
\begin{equation*}
\pi^c_2:=\int_0^{T^f}e^{-rt}f^c\big(t,S^c(t),q^c(t)\big)dt+\
\int_{T^f}^{T^c}e^{-rt}f^c\big(t,S^c(t),q^c(t)\big)dt.
\end{equation*}

Let $I_1^c=\int_0^{T^f}e^{-rt}f^c\big(t,S^c(t),q^c(t)\big)dt$ and $I^c_2=\int_{T^f}^{T^c}e^{-rt}f^c\big(t,S^c(t),q^c(t)\big)dt$. We have shown that in the phase where both players extract resources simultaneously, the optimal cartel control $q^c(t)$ is $\frac{\alpha-k^c-M^f\beta}{2\beta}-\frac{(b-k^c)}{2\beta}\cdot e^{r(t-T^f-T^c)}$. In the case where the cartel remains the only player in the oil market, we have $q^c(t)=\frac{\alpha-b}{\beta}$.

Below, we have calculated the components of the integral $\pi^c_2$, namely $I_1^c$ and $I_2^c$. Let us start by computing the integral $I_1^c$. Here is the explicit form of the integral $I_1^c$:
\begin{equation*}
\begin{aligned}
I_1^c = \int_0^{T^f} e^{-rt} \, 
& \Biggl(\frac{\alpha + k^c + M^f \beta - 2 k^c}{2} 
   + \frac{(b - k^c) e^{\, r(t - T^f - T^c)}}{2} \Biggr) \\ 
& \cdot \Biggl(\frac{\alpha - k^c - M^f \beta}{2 \beta} 
   - \frac{b - k^c}{2 \beta} e^{\, r(t - T^f - T^c)} \Biggr) \, dt.
\end{aligned}
\end{equation*}

After evaluating the integral of the exponential function, we obtain:
\begin{equation*}
I_1^c=-\frac{\alpha-k^c-M^f\beta}{2\beta r}\cdot\frac{\alpha-k^c+M^f\beta}{2}\cdot e^{-rt}\Big|_0^{T^f}+\frac{b-k^c}{2}\cdot\frac{\alpha-k^c-M^f\beta}{2\beta}\cdot e^{-r(T^f+T^c)}\cdot t \Big|_0^{T^f}-
\end{equation*}
\begin{equation*}
-\frac{\alpha-k^c+M^f\beta}{2}\cdot \frac{b-k^c}{2\beta}\cdot e^{-r(T^f+T^c)}\cdot t \Big|_0^{T^f}+\frac{1}{r \beta}\Big(\frac{b-k^c}{2}\Big)^2 e^{r(t-2T^f-2T^c)}\Big|_0^{T^f}.
\end{equation*}
After simplifying, we obtain:
\begin{equation*}
I_1^c=\frac{\alpha-k^c-M^f\beta}{2\beta r}\cdot\frac{\alpha-k^c+M^f\beta}{2}\cdot\Big(1-e^{-rT^f}\Big)+\frac{b-k^c}{2}\cdot\frac{\alpha-k^c-M^f\beta}{2\beta}\cdot e^{-r(T^f+T^c)}\cdot T^f-
\end{equation*}
\begin{equation*}
-\frac{\alpha-k^c+M^f\beta}{2}\cdot \frac{b-k^c}{2\beta}e^{-r(T^f+T^c)}\cdot T^f+\frac{1}{r \beta}\Big(\frac{b-k^c}{2}\Big)^2 e^{-r(T^f-2T^c)}\Big(1-e^{-rT^f}\Big).
\end{equation*}

Now, we calculate $I_2^c$, taking into account the optimal level of cartel control in this period:
\begin{equation*}
I_2^c=\frac{1}{r}(b-k^c)\cdot\frac{\alpha-b}{\beta}\cdot\big(e^{-rT^f}-e^{-rT^c}\big),
\end{equation*}
where $T^f=\frac{S_0^f}{M^f}$ and $T^c$ is the solution to equation \eqref{rownanie_z_tece}. Calculating the net profit $\pi^c_2$ requires summing the two components $I_1^c$ and $I_2^c$. Therefore,
\begin{equation*}
\pi^c_2=\frac{\alpha-k^c-M^f\beta}{2\beta r}\cdot\frac{\alpha-k^c+M^f\beta}{2}\cdot\Big(1-e^{-rT^f}\Big)+\frac{b-k^c}{2}\cdot\frac{\alpha-b-M^f\beta}{2\beta}\cdot e^{-r(T^f+T^c)}\cdot T^f-
\end{equation*}
\begin{equation*}
-\frac{\alpha-k^c+M^f\beta}{2}\cdot \frac{b-k^c}{2\beta}e^{-r(T^f+T^c)}\cdot T^f-\frac{1}{r \beta}\Big(\frac{b-k^c}{2}\Big)^2 e^{-rT^f-2T^c}\Big(1-e^{-rT^f}\Big)+
\end{equation*}
\begin{equation}\label{zysk2}
+\frac{1}{r}(b-k^c)\cdot\frac{\alpha-b}{\beta}\cdot\big(e^{-rT^f}-e^{-rT^c}\big).
\end{equation}

\subsubsection{Profit for optimal over ''Wait'' class of strategies (see page \pageref{def_scenarios})}\label{s3}
We calculate the profit for the third cartel strategy. For this strategy, finding the result involves computing the integral:
\begin{equation*}
    \pi^c_3:=\int_{T^f}^{\hat{T}^c}e^{-rt}\big(\alpha-\beta \cdot \frac{\alpha-b}{\beta}-k^c\big)\cdot\frac{\alpha-b}{\beta} dt.
\end{equation*}
After evaluating the integral, we obtain:
\begin{equation*}
    -\frac{1}{r}e^{-rt}(b-k^c)\cdot\frac{\alpha-b}{\beta} \Big |_{T^f}^{\hat{T}^c}=\frac{1}{r}(b-k^c)\cdot\frac{\alpha-b}{\beta}\cdot(e^{-rT^f}-e^{-r\hat{T}^c}),
\end{equation*}
where $\hat{T}^c= T^f+\frac{S_0^c \cdot \beta}{\alpha - b}$ and $T^f=\frac{S_0^f}{M^f}$. Therefore, the net profit in the third case is:

\begin{equation}\label{zysk23}
\begin{aligned}
\pi^c_3 
&= \frac{1}{r} (b - k^c) \cdot \frac{\alpha - b}{\beta} 
   \cdot \Biggl( e^{-r T^f} - e^{-r \Bigl(T^f + \frac{S_0^c \cdot \beta}{\alpha - b}\Bigr)} \Biggr) \\[1mm]
&= \frac{1}{r} (b - k^c) \cdot \frac{\alpha - b}{\beta} \cdot e^{-r T^f} 
   \Biggl( 1 - e^{-r \left(\frac{S_0^c \cdot \beta}{\alpha - b}\right)} \Biggr).
\end{aligned}
\end{equation}

\subsubsection{Conclusions based on calculated profits}\label{wnioski}

By comparing the explicit forms of the discounted net profits obtained for each strategy, given by \eqref{zysk1}, \eqref{zysk2} and \eqref{zysk23}, it is evident that they depend collectively on the following set of parameters: $k^c, k^f, b$, $\alpha, S_0^c, S_0^f, r$ and $M^f$. We arrive at the following conclusion: for the cartel to choose its optimal strategy, it must have knowledge of values of all the parameters included in the model, including those that control the actions of the rival extractors. Therefore, full information is necessary. To select the optimal cartel strategy, one could calibrate the model based on the most current data available (similarly to the approach taken in the work of \cite{BENCHEKROUN2020102384}, where, for example, data from the Energy Information Administration for the year 2013 were utilized). 

Instead of calibrating the model and obtaining static results, we can also conduct a sensitivity analysis based on the values of individual parameters. Thus, in the next section, we perform a graphical sensitivity analysis of the calculated discounted net profits obtained for each strategy.

\subsection{Graphical analysis}\label{s_graphical}

In this section, we present a graphical sensitivity analysis of cartel’s discounted net profits obtained from selecting
given strategy. The strategies are characterized in Subsection \ref{podrozdzial_strategie}. We use the following initial set of parameter values that constitute a starting point of our analysis. Further, we calibrate the model for manipulating set of chosen parameters at once. To this end, figure captions will include in parentheses the values of parameters that differ from the benchmark. If no values are given in parentheses, this means that the benchmark calibration is used.

\newpage
\begin{table}\label{table_benchmark}
\centering
\begin{tabular}{ |p{3cm}||p{6cm}|p{1.5cm}|p{1.5cm}|  }
 \hline
 \textbf{Parameters} & \textbf{Description} &\textbf{Value}&\textbf{Unit}\\
 \hline
 $\alpha$   & choke price    &225.5&   US\$/bbl\\
 $\beta$&   slope inverse demand function  & 4.3   &US\$/bbl\\
 $b$ & renewables price & 102.5&  US\$/BOE\\
 $k^c$   & marginal extraction cost of the cartel  & 18 &  US\$/bbl\\
 $k^f$ &   marginal extraction cost of the fringe  & 62.5 & US\$/bbl\\
 $r$ & interest rate  & 0.028   & percentage\\
 $S^c_0$ & initial stock of the cartel  & 1212 & billion bbl\\
 $S^f_0$ & initial stock of the fringe  & 619.5 & billion bbl\\
 \hline
 
\end{tabular}
\centering \caption{Benchmark ("status quo") calibration}
\end{table}

For the next three figures, the maximal depletion time of the cartel---being the end point of the analysis---was calculated numerically and is the solution to equation \eqref{rownanie_z_tece}.

Figure \ref{graph_q} illustrates the relationship between the optimal control of the cartel corresponding to a given strategy and the duration of the game, with the vertical lines representing discontinuity points. From now on, we can interpret optimal controls as extraction rates (as interpreted in \cite{BENCHEKROUN2020102384}). The length of the horizontal axis represents the duration of the nonrenewable resource game
(in years) and the length of the vertical axis represents the value of cartel’s optimal control in Strategy $i=1,2,3$, which is a dimensionless quantity. We can see that after about 64.5 years the resource of the cartel is depleted. Clearly, in case of the first and third
strategy, the value of cartel’s optimal control is constant over time corresponding to a given strategy. In case of the second strategy, the optimal control of the cartel exponentially decreases over time to the depletion time of the fringe, given by \eqref{teef}.
Afterwards, as cartel's optimal control strategy is a piecewise function, its value jumps to a constant value $\frac{\alpha-b}{\beta}$ corresponding to the price $p(t)=b$ for $t \in L$ (see: Notation \ref{ozn_fazy}). 

On the other hand, Figure \ref{graph_p} illustrates the relationship between price functions corresponding to the optimal cartel controls and the fringe in the given strategy. The vertical lines mark points of discontinuity. Obviously, in the case of the first and third strategies, the price function remains constant, which follows from the price-related assumptions underlying these strategies. In case of the first strategy, in the moment of cartel's depletion, the price jumps to $\min\left(\alpha-\beta M^F, b\right)=b$, as the value of $M^f$ is relatively small. In the second strategy, the price increases exponentially over time to the fringe depletion moment and then is constant and equal to $b$. 

Figure \ref{graph_whole_profit} shows the present value of accumulated profits up to time $t$ resulting from undertaking each type of strategy (which are functions of optimal controls and price paths). Each curve is an increasing function as we consider accumulated profits with respect to time. Discernibly, the second strategy yields the highest profit, approximately 53,500, followed by the first strategy with about 35,600. The lowest discounted profit, around 32,300, is implied by the third strategy.

\begin{figure}[!htb]
    \centering 
\includegraphics[width=10cm]{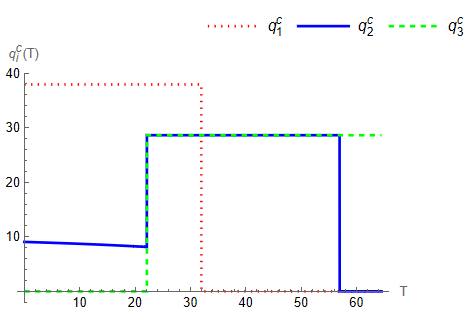}
\caption{Cartel’s  optimal control in three types of strategies: 1. ''Deter'', 2. ''Share'', 3. ''Wait''}\label{graph_q}\
\end{figure}

Now, we analyze how sensitive the discounted net profit of the cartel is to the change of following parameters: $M^f$, $S^f$, $k^f$ and $k^c$.

Figure \ref{graph_pi_mf} shows how a change in the parameter $M^f$, being the optimal control of the fringe within a simultaneous extraction phase, affects the value of the cartel's net profit corresponding to a given strategy. The value of $M^f$ is strictly bounded by $\frac{\alpha-b}{\beta}$, what follows from imposed condition \eqref{odcinek_zbior} on the set of control parameters of the cartel. Using the data concerning calibration of the model from \cite{BENCHEKROUN2020102384}, we obtain $\frac{\alpha-b}{\beta}\approx 28.6$, thus we bound the horizontal axis to a value of 28.5. We can observe that within the first strategy, the value of the $\pi^c_1(M^f)$ is independent of the $M^f$ value, so the former remains constant. On the other hand, the value of cartel's discounted net profit in case of the second and third strategy is a decreasing and increasing function of $M^f$, respectively. However, regardless of the value of $M^f$ lying between 0 and 28.5, cartel's net profit obtained from choosing: the second strategy is the highest, the first strategy is the second highest and the third strategy is the lowest.

Figure \ref{graph_pi_mf_r_sens} (left) shows the relationship between the change of $M^f$ parameter and the net profit of the cartel, when the value of an interest rate is set to be almost ten times lower than before (we set $r=0.2 \%$). We can observe that in such market conditions with $M^f$ values being remarkably high ($25.5$ or more), we get an unusual result---the best strategy for the cartel is to wait until the fringe depletes its resource and after that the cartel starts mining its resource and selling it at price $b$ (the close-up of intersection point is shown in Figure \ref{graph_pi_mf_high_r_closeup}, left). The profit resulting from taking the third strategy is highly sensitive on the change of $r$ parameter. When the value of $M^f$ is about $2.0$, the net profit from the third strategy is even the lowest from all three options. After rapid growth of considered profit function for relatively small values of $M^f$, its growth becomes milder for higher values of $M^f$.

For transparency of exposition, we now present a graph depicting a situation, where the strategy "Deter" is the optimal one. Figure \ref{graph_pi_mf_r_sens} (right) shows the relationship between the change of $M^f$ parameter and the net profit of the cartel, when the value of an interest rate is set to be almost ten times higher than before (we set $r=20 \%$) i.e., when agents are highly impatient. We can observe that when the parameter equal to the optimal control of the fringe is close to its upper bound (close to 28.6), the first strategy "Deter" indeed prevails over the second and third strategies, which is further shown in Figure \ref{graph_pi_mf_high_r_closeup} (right).   

\begin{figure}[!htb]
    \centering 
\includegraphics[width=10cm]{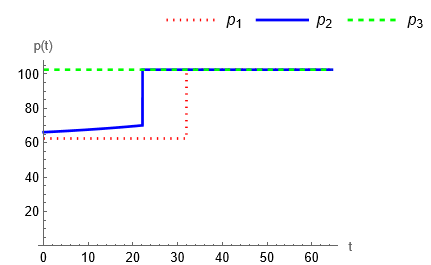}
\caption{Price corresponding to cartel’s controls optimal in three classes of strategies: 1. ''Deter'', 2. ''Share'', 3. ''Wait''}\label{graph_p}
\end{figure}

\begin{figure}[!htb]
    \centering 
\includegraphics[width=10cm]{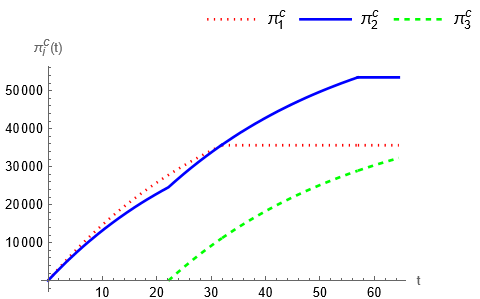}
\caption{Cartel’s optimal accumulated profit over time in three classes of strategies: 1. ''Deter'', 2. ''Share'', 3. ''Wait''}\label{graph_whole_profit}\
\end{figure}

\begin{figure}[!htb]
    \centering 
\includegraphics[width=10cm]{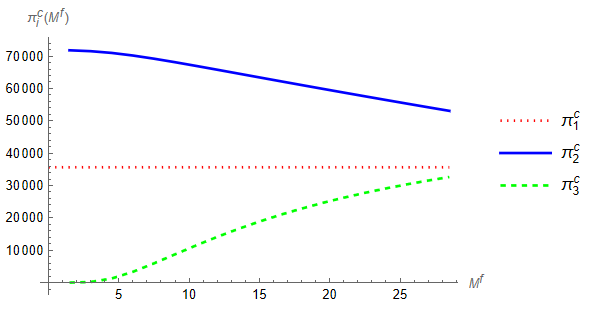}
\caption{Relationship between cartel's net profit and value of $M^f$ parameter (for $r=0.028$) }\label{graph_pi_mf}
\end{figure}

\begin{figure}[!htb]
    \centering 
\includegraphics[width=6.5cm]{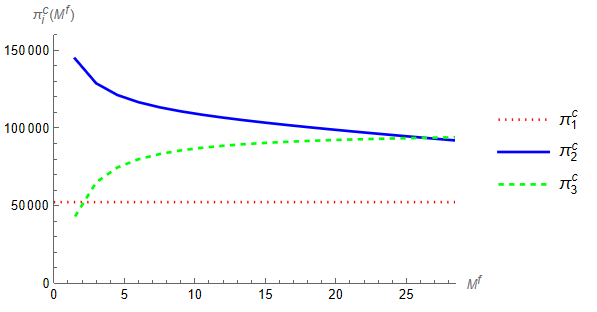}
\includegraphics[width=6.5cm]{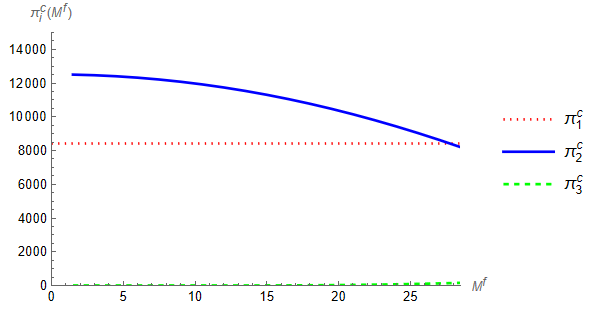}
\caption{Relationship between cartel's net profit and value of $M^f$ parameter (left: for $r=0.002$, right: for $r=0.2$ )}\label{graph_pi_mf_r_sens}
\end{figure}

\begin{figure}[!htb]
    \centering 
\includegraphics[width=6.5cm]{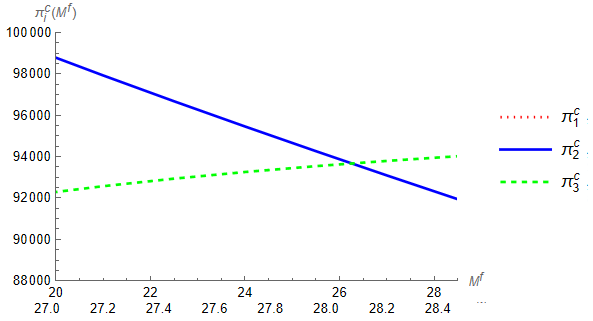}
\includegraphics[width=6.5cm]{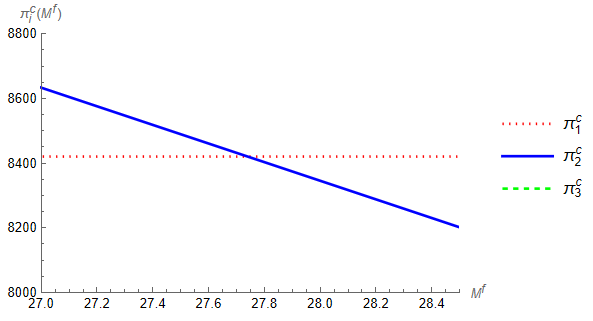}
\caption{Close-up on the intersection point of cartel's net profits (left: for $r=0.002$, right: for $r=0.2$)}\label{graph_pi_mf_high_r_closeup}
\end{figure}


\begin{figure}[!htb]
    \centering 
\includegraphics[width=6.7cm]{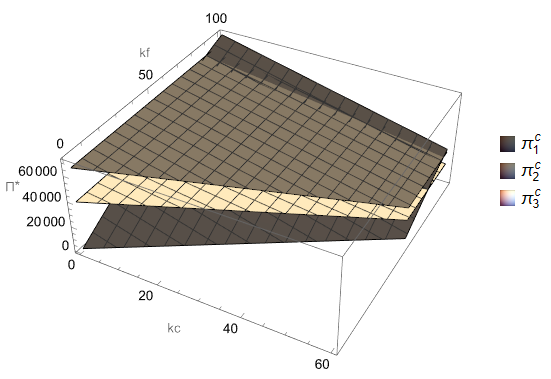}
\includegraphics[width=6.7cm]{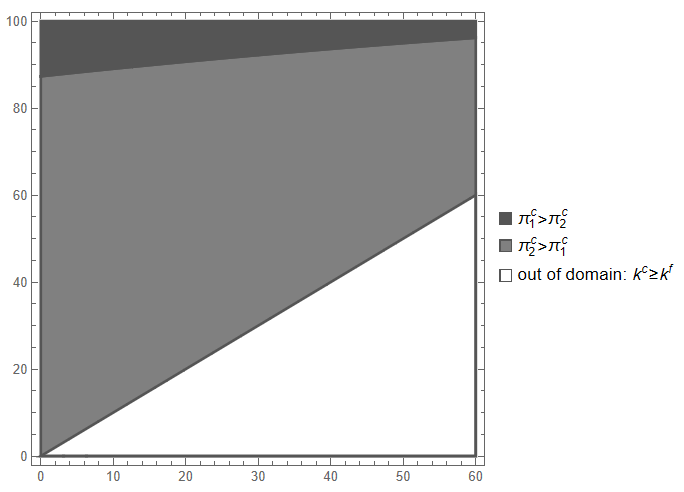}
\caption{Relationship between cartel's net profit and value of $k^f$ (range $0-100$) and $k^c$ (range 0-60) (left: 3-d plot, right: region plot)}\label{graph_pi_kc_kf}
\end{figure}

Figure \ref{graph_pi_kc_kf} (left) illustrates the relationship between the profit of the cartel corresponding to a given strategy and the marginal extraction cost: $k^f$ and $k^c$. We account for the assumption posed in Section \ref{s_formulation} that $k^f<k^c$ by excluding from the domain of cartel's profit function the half-plane $\{(k^c,k^f,\pi^c): k^c\geq k^f\}$. Interestingly, the discounted net profit of the cartel in case of the first strategy can be the highest, when the value of $k^f$ is relatively high to the marginal cost of the fringe, i.e. $k^f$. Thus, cartel's net profit is highly sensitive to the relation between $k^f$ and $k^c$. Based on this plot, we also conclude that cartel's net profit value is highly sensitive of $k^f$ and not that much on the value of $k^c$. Moreover, regardless of the value of marginal costs, the profit obtained from choosing the second strategy is higher than the profit obtained from the third strategy. 

For the purpose of a more detailed marginal costs profile of cartel's profit, we analyze Figure \ref{graph_pi_kc_kf} (right) -- this plot shows the regions where the profit obtained from choosing strategy $i$ is higher than one from choosing strategy $j$, such that $i,j=1,2,3$ and $i\neq j$. The numerical relationship between $k^f$ and $k^c$, resulting in the highest profit obtained from strategy $i$, is disclosed in Appendix \ref{s_graphical_appendix}.

Figure \ref{graph_pi_s0f} illustrates the impact of $S^f_0$ (the initial resource of the fringe) on the cartel's net profit. We do not conduct sensitivity analysis with respect to $S^c_0$, as the cartel is aware of its initial resource level. We can see that when a resource level is relatively low, then the third strategy, i.e. wait until the fringe depletes and then extract and sell at the price $b$ results in higher profit than the first strategy, i.e. do not allow the fringe to enter the market and sell your product at price $k^f$ until your resources are depleted. Nevertheless, both of these strategies are outperformed by the second strategy, i.e., first extract your resources simultaneously and after fringe's depletion, continue extracting but sell the product at the price $b$. The profit resulting from choosing the second strategy nearly does not change with the change of $S^f_0$ parameter. Thus, the size of joint resource of the fringe in considered range, that is their market power, almost does not affect the profit -- cartel's net profit is low-sensitive to the change of fringe's initial stock. Should we consider a bigger initial resource stock of the fringe, it would not be much realistic as the fringe has significantly smaller resource level (in particular it should not be higher than the initial resource level of the cartel).

\begin{figure}[!htb]
    \centering 
\includegraphics[width=10cm]{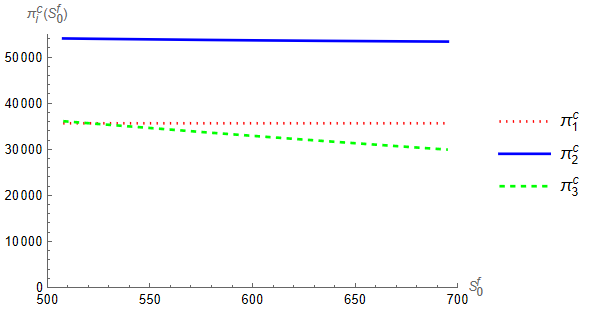}
\caption{Relationship between cartel's net profit and value of $S_0^f$ (the initial level of fringe's resource stock).}\label{graph_pi_s0f}
\end{figure}

\section{Conclusions and further research}\label{s_conclusion}
\subsection{Conclusions}
In this study, we have presented a model of renewable resource extraction with one cartel and a continuum of fringe extractors and a backstop technology substantially more expensive than cost of resource extraction from the most costly deposit. This differs from the standard cartel-fringe interaction with fringe treated as one decision maker. Resigning from the standard ``representative fringe firm'' approach allowed us to both model interactions among the fringe firms and model them as rational agents. In our model, the cartel has a separate deposit, while the fringe members pump from the same deposit, although each of them has a separate well. We have obtained myopia of the fringe as the result of ``the tragedy of the commons'': if they decided cooperatively, they would reduce current extraction, which would increase the price and consequently, profit, but for each single fringe firm, extracting more increases current profit as it has only small influence on price, while the cost of decreasing the resource stock is shared between all fringe members.  Besides, we relaxed the standard price continuity assumption and we imposed constraints on strategies.
We have calculated the Stackelberg equilibrium of this game with the cartel as the leader and the fringe as multiple followers. We obtained that the joint reaction function of fringe firms is unique and discontinuous despite the components of the model besides those mentioned above are similar to those analyzed in the literature, and identical with \cite{BENCHEKROUN2020102384}. 
As far as final results are concerned, it is typical for cartel-fringe extraction models in the literature that both firms extract at the same stage until the fringe depletes. However, besides this typical case, we obtained two nonstandard results: for some parameters values, it is better for the cartel to wait until the fringe depletes and then start extracting, while for some other parameters values, it is better for the cartel to set a dumping price to deter the higher-cost fringe from extraction.

\subsection{Immediate extensions and  further research}
There are two immediate model extensions. First, we can introduce increasing or decreasing returns to scale of the cartel's technology by subtracting a quadratic component from the costs, $\frac{\xi^c}{2} \left(q^c\right)^2$, with $\xi$ positive or negative, respectively. Then, the results do not change substantially unless this changes the property that the cartel's current payoff, taking into account also the influence on price, remains concave in $q_i^c$. Second, we can consider increasing returns to scale of the fringe in an analogous way, by subtracting a quadratic component from the costs, $\frac{\xi^f}{2} \left(q_i^f\right)^2$. This changes only the cartel's strategy and payoff in the ''deter'' strategies, as the required maximal price changes to $k^f-\xi^f\max_i\{m^f_i\}$. What is new is the fact that it may be too costly to deter all fringe members if their pumping constraints are not identical.

Our model can be a source for further research.  


A further recommendation for model development is to introduce the possibility of storing the extracted goods. Similarly to \cite{benchekroun2015harvests}, one may consider that each of the players can store their extracted oil at some storage costs, with constraints on the storage capacity and/or costs of building such an oil terminal. If the cartel (leader) notices that the fringe limits sales and stores its extracted oil, then it faces market disequilibrium and should recalculate its optimization problem, taking into account the storage cost of the fringe. However, the followers have limited storage capacity, which leads to a small change in value of cartel's net profit. If the storage capacity is small, the market price can still remain $b$ after the depletion of fringe's resource if the cartel slightly delays its extraction.

One could also consider a stochastic control problem for this model. The modeling of resource stocks level under this framework could address the aspect of uncertainty due to disturbances occurring during the process of extracting nonrenewable resources. Therefore, stochastic differential equations could be a good framework for modeling resource extraction with potential problems with the extraction process itself \citep[see Chapter 3.4 in][]{yong2012stochastic}.

On the other hand, in the considered model one could assume that the indifference price $b$ is time dependent. In this framework, the function $b(t)$  could be exponential and decreasing, as proposed by e.g. \cite{JAAKKOLA2019101}. The latter assumption is justified---over time, the firms can develop more efficient sourcing technologies, which leads to increased competitiveness of a product.

Another extension is to add state-dependent constraints on controls for both the cartel and the fringe for low state of their deposit: $q_i^f(t)\leq \min\{m_i^f,\nu_i^f S^f(t)\}$, $q^c(t)\leq \min\{m^c, \nu^c S^c(t)\}$ (where $\nu_i^f$ and $\nu^c$ are some positive constants). This makes full depletion of the resource impossible at any finite time due to technical constraints. An analogous result but due to economic constraints can be obtained for allowing the marginal cost to be decreasing in state, especially for state values close to 0. 

We can also decrease the number of the fringe members using the commons deposit and then take into account the spatial allocation of the fringe's resource with flows among wells of various players determined by adjacency matrix, as in e.g. \cite{fabbri2024competition}. In this case, the fringe behavior is no longer myopic.

\section*{Acknowledgments}
This research was financed by a grant 2019/33/B/HS4/02953 of the National Science Center, Poland.

\section*{Declarations}

The authors declare no conflict of interest.

\newpage



\begin{appendices}

\setcounter{equation}{0}
\renewcommand{\theequation}{\Alph{section}\arabic{equation}}

\section{Methodology}\label{s_methodology}
The presence of the fringe, whose decisions may be discontinuous, result in discontinuity of the functions defining the optimal control problem of the cartel. Therefore, the standard usual tools on existence of the optimal control and the necessary conditions do not work. Thus, to prove the results, we decompose the initial problem into two coupled problems: a fixed time horizon free endpoint optimal control problem and free terminal time defined by terminal constraint. 

To do this, we introduce two versions of the Pontryagin's Maximum Principle (abbreviated as PMP) that we use in further analysis: first for the problem with free terminal time and fixed terminal state, and latter for the problem with fixed terminal time and free terminal state. Then, present the theorem about the existence of optimal control.

\subsection{Tools}
In the statements provided below, we consider the Bolza problem to maintain consistency with the later part of the work. The general form of the Bolza problem with a fixed terminal time $T_1$
is derived from the work of \cite{bressan2007introduction}, but we present it in a simplified version.

Let $x(t) \in \mathbb{R}$ and $f:\mathbb{R}^3 \rightarrow \mathbb{R}$, $\varphi:\mathbb{R}^3\ \rightarrow \mathbb{R}$, $g:\mathbb{R}^2 \rightarrow \mathbb{R}$. We maximize the payoff given by the functional

\begin{equation}\label{Bolza_funkcjonal}
    J\big(x(\cdot),u(\cdot)\big)=\int_0^{T_1}e^{-rt}f\big(t, x(t), u(t)\big)dt+g\big(T_1,x(T_1)\big)e^{-rT_1},
\end{equation}
subject to
\begin{equation}\label{bolza_zmianastanu}
    \dot{x}(t)=\varphi\big(t, x(t), u(t)\big) \quad\textit{a.e.}
\end{equation}
with the initial condition $x(0)=x_0$,\\
where $r>0$ is continuous compounding interest rate.

Additionally, there are constraints on control parameters $u(t) \in [0,M]$ for all $t$, where $M>0$. 
We assume that control functions $u$ are measurable and the corresponding trajectory $x$ is absolutely continuous. 

In the PMP theorem with free terminal time and fixed terminal state, we consider constraints on the terminal state. These constraints are posed on control parameters for which there exists a depletion time of resource $T$ and function $\psi:\mathbb{R}^2 \rightarrow \mathbb{R}$ such that
\begin{equation}\label{bolza_stan}
    \psi\big(T^*,x(T^*)\big)=0.
\end{equation}
In the considered case $x(T^*)$ represents the terminal state of resources for a given agent. Moreover, $T^*$ is the first such moment in the game when the equation \eqref{bolza_stan} is satisfied.

Before we introduce two versions of PMP, we define present value of Hamiltonian function.
\begin{definition}\label{definicja_hamiltonian}
    A Hamiltonian $H: \mathbb{R}_+ \times \mathbb{R} \times [0,M] \times \mathbb{R}\times \mathbb{R}_+ \rightarrow \mathbb{R}$ is the function 
    \begin{equation*}
    H(t,x,u,\gamma, \lambda_0):=\lambda_0 e^{-rt} f(t,x,u) + \langle \gamma,\varphi(t,x,u) \rangle,
\end{equation*}
where $\gamma \in \mathbb{R}$ and $\lambda_0 \in \mathbb{R}_+$. In case when $\lambda_0=1$, we omit $\lambda_0$ for simplicity's sake.
\end{definition}

\subsubsection{Pontryagin's Maximum Principle with free terminal time and fixed terminal state}\label{appendix_fixed_state}

The Maximum Principle for the Bolza problem with free terminal time and fixed terminal state is derived from \cite{bressan2007introduction}, labeled there as Theorem 6.5.1. In this work, we have presented this theorem in the maximization version and in a simplified form.

\begin{theorem}[PMP with free terminal time and fixed terminal state]\label{ZMP_zwiazane}
Let the functions $f,\varphi$ be continuous with respect to all their variables and of class $C^1$ with respect to $(t,x)$. Let the control $u_*$ 
in the Bolza problem \eqref{Bolza_funkcjonal}-\eqref{bolza_stan} be bounded and let the gradient $\nabla\psi\left(T^*,x(T^*)\right)$ be nonzero. 
Then there exists a nontrivial, absolutely continuous function $\gamma$ and constants $\lambda_0 \geq 0,\lambda_1 \in \mathbb{R}$ such that for almost all $t \in [0,T^*]$
\\
\begin{equation}\label{zmp_zwiazane_1}
    \dot{\gamma}(t)=-\frac{\partial H\big(t, x_*(t), u_*(t), \gamma(t), \lambda_0\big)}{\partial x},
\end{equation}

 \begin{equation}\label{argmax_u}
u_*(t) \in \Argmax_{u \in [0,M]}H\big(t, x_*(t), u,\gamma(t)\big),
\end{equation}
 \begin{equation}
   \gamma(T^*)=\lambda_1 \frac{\partial \psi}{\partial x}\big(T^*,x_*(T^*)\big),
\end{equation}
\begin{equation}\label{max_zero}
    \max_{u \in [0,M]} H\big(T^*,x_*(T^*),u,\gamma(T^*),\lambda_0\big)=-\lambda_1 \frac{\partial \psi}{\partial t}\big(T^*,x_*(T^*)\big),
\end{equation}
\begin{equation}
    \frac{d}{dt}H\big(t,x_*(t),u_*(t),\gamma(t),\lambda_0\big)=\big\langle \gamma, \frac{\partial }{\partial t} \varphi\big(t,x_*(t),u_*(t)\big)\big \rangle + \lambda_0\frac{\partial}{\partial t} f\big(t,x_*(t),u_*(t)\big).
\end{equation}
\end{theorem}
\subsubsection{Pontryagin's Maximum Principle with fixed terminal time and free terminal state}\label{appendix_free_State}

In the context of optimizing control for phases of extraction that are not the terminal ones in the market game, we introduce the maximum principle for the Bolza problem with fixed time and free terminal state, stated by equations \eqref{Bolza_funkcjonal}-\eqref{bolza_zmianastanu}. This form of the Bolza problem was formulated e.g. in \cite{zabczyk2020mathematical}. 

\begin{theorem}[PMP with fixed terminal time and free terminal state]\label{zmp_zwykla}
    Let $f$, $g$ i $\varphi$ and their partial derivatives with respect to $x$ be continuous functions on their domains. 
    \setlength{\parindent}{20pt}If $u_*$ is the control maximizing the payout $J\big(x(\cdot),u(\cdot)\big)$, and $x_*$ be the trajectory determined by $u^*$, then in Definition \ref{definicja_hamiltonian} occurs $\lambda_0=1$ and there exists an absolutely continuous function $\gamma:[0,T_1] \rightarrow \mathbb{R}$ such that for all $t$, in which exists a left-sided derivative $x_*(t)$ equal to $\varphi\big(t,x_*(t),u_*(t)\big)$ occurs:
    \begin{equation}
    x_*(t)=\frac{\partial H\big(t,x_*(t),u_*(t),
    \gamma(t)\big)}{\partial \gamma} \text{ with the initial condition } x_*(0)=x_0,
    \end{equation}
    \begin{equation}\label{zmp_2_wk}
    \dot{\gamma}(t)=-\frac{\partial H\big(t,x_*(t),u_*(t),
    \gamma(t)\big)}{\partial x} \text{ with the terminal condition } \gamma(T_1)=\frac{\partial g\big(t,x_*(T_1)\big)}{\partial x}e^{-rT_1}, 
    \end{equation}
    \begin{equation}\label{argmax_u_zwykla}
    u_*(t) \in \Argmax_{u \in [0,M]}H\big(t,x_*(t),u,\gamma(t)\big).         
    \end{equation}

\end{theorem}

\subsubsection{Theorem on the existence of optimal control}

At the end of this methodological chapter, we introduce the theorem about the existence of optimal control, which is a solution to the Bolza problem. We extract it from the work of \cite{bressan2007introduction}. However, we present it in versions applicable to two simpler maximization problems: one described by equations \eqref{Bolza_funkcjonal}-\eqref{bolza_stan} and the other described by equations \eqref{Bolza_funkcjonal}-\eqref{bolza_zmianastanu}. Denote one of the assumptions of the discussed theorem as $(\mathbb{\hat{H}})$:
\\

  $(\mathbb{\hat{H}})$ Let $[0,M]$ be a set of control functions such that $[0,M] \in [0,M]$ for some constant $M>0$. Furthermore, let $\varphi:[0,\infty)\times\mathbb{R}\times[0,M] \rightarrow \mathbb{R}$ be a continuous function with respect to all its variables and of class $C^1$ with respect to trajectory $x$ such that it satisfies the constraint
  \begin{equation}
      |\varphi(t,x,u)|\leq C(1+|x|)
  \end{equation}
  for some constant $C>0$.

  Besides, we need to define family of target sets $\mathbb{S}$.

\begin{definition}
    Let $S(t)$ be a family of sets such that for the problem in form \eqref{Bolza_funkcjonal}-\eqref{bolza_stan} is given by
    \begin{equation}\label{celu_swobodny}
        \mathbb{S}(t)=\{x: \psi(t,x)=0\},
    \end{equation}
    and for the problem in form \eqref{Bolza_funkcjonal}-\eqref{bolza_zmianastanu} with fixed terminal time $T_1$
    \[\mathbb{S}(t)=
    \begin{cases}
    \mathbb{R} \quad \quad t=T_1& \\
    \emptyset \quad \quad \text{otherwise}.            &
\end{cases}
\]  
\end{definition}

\begin{theorem}[Existence of optimal control in Bolza problem]\label{tw_o_istnieniu}
Assume that $\varphi$, $\mathbb{U}$ satisfy the condition $(\mathbb{\hat{H}})$. Let $\mathbb{S}(t)$ be a family of closed target sets in a problem of the form \eqref{Bolza_funkcjonal}-\eqref{bolza_stan} or of the form \eqref{Bolza_funkcjonal}-\eqref{bolza_zmianastanu}. Assume that there exists a control $u$ and the corresponding trajectory $x$ satisfying \eqref{bolza_zmianastanu} such that $x(\overline{T}) \in \mathbb{S}(\overline{T})$ for some $\overline{T}$. Furthermore, let $f$ be a continuous function. If $g$ is a continuous function and the sets of the form
\begin{equation}\label{zbiory_wypukle}
    F^-(t,x)=\{(y_0,y) \in \mathbb{R}^2: y_0 \leq f(t,x,u)e^{-rt}, \quad y=\varphi(t,x,u) \quad \text{for } u \in \mathbb{U}\}
\end{equation}
are convex, then the problem \eqref{Bolza_funkcjonal}-\eqref{bolza_stan} with free terminal time or the problem \eqref{Bolza_funkcjonal}-\eqref{bolza_zmianastanu} with fixed terminal time has an optimal solution (namely, there exists an optimal control).
     
\end{theorem}


\subsection{Technical proofs}\label{appendix_tech_proofs}

\subsubsection{Proof of the existence of the optimal control for the cartel}\label{podrozdzial_weryf_tw_o_istn}
To prove the existence of the optimal control, it is necessary to positively verify the assumptions of Theorem \ref{tw_o_istnieniu}. These assumptions need to be checked for the two problems considered in the paper: the maximization problem with free terminal time and fixed terminal state, described by equations \eqref{Bolza_funkcjonal}-\eqref{bolza_stan}, and the maximization problem with fixed terminal time and free terminal state, described by equations \eqref{Bolza_funkcjonal}-\eqref{bolza_zmianastanu}

\begin{proposition}
The assumptions of Theorem \ref{tw_o_istnieniu} are satisfied for the problem with free terminal time and fixed terminal state described by equations \eqref{Bolza_funkcjonal}-\eqref{bolza_stan} without switching between phases.
\end{proposition}

\begin{proof}
 Consider the optimization problem for the cartel with free terminal time, described by equations \eqref{Bolza_funkcjonal}-\eqref{bolza_stan}. This problem corresponds to the last phase in each of strategy classes deter, share or wait which has been described in Subsection \ref{podrozdzial_strategie}. 

 Consider the target set given by formula \eqref{celu_swobodny}. Certainly, the target set $\mathbb{S}$ is attainable for the trajectory $S^c$ given by \eqref{zmiana_st_c}. The price-setting cartel can set the price at any level higher than $k^c$ and in that case, it will generate profits. In this phase, there is a lower bound set on cartel's control parameter, as stated in the set of inequalities given by \eqref{scen3_przedzial_par_ster}: $q^c>\frac{\alpha-\beta M^f-b}{\beta}$. This means that the resource state decreases at a rate greater than zero at any moment during the game. This proves the existence of a moment $T^c$ such that $S^c(T^c)=0$, implying $S^c(T^c) \in \mathbb{S}(T^c)$.

 When it comes to the continuity and differentiability conditions of the functions $f^c$ and $\varphi^c$, it is sufficient to show that the functions
 \begin{itemize}
    \item $f^c(t,S^c,q^c)=(\alpha-\beta(Q^f+q^c)-k^c)q^c$,
    \item $\varphi^c\big(t,S^c\big)=-q^c(t)$
\end{itemize}
are continuous with respect to all variables and $C^1$ with respect to $(t,S^c)$. Note that none of the functions $f^c, \varphi^c$ depend explicitly on $S^c$. Each of them depends quadratically and linearly on $q^c$, so the statement is obvious. 

Now, we show that $\varphi^c$ satisfies $(\mathbb{\hat{H}})$. We need to find a constant $C>0$ such that
\begin{equation*}
    |-q^c(t)|\leq C\Big(1+\big|-\int_t^{T^c}q^c(t)\big|\Big).
\end{equation*}
As the constraint $p(t)=\alpha-\beta q^c \geq 0$ implies $|q^c|\leq\frac{\alpha}{\beta}$, it is sufficient to take $C=\frac{\alpha}{\beta}$. 

Finally, it needs to be shown that the sets $F^-(t,x)$ defined as in equation \eqref{zbiory_wypukle} are convex. Consider points $(y_0,y), (z_0,z) \in F^-(t,x)$. We aim to prove that the convex combination $(ay_0+(1-a)z_0, ay+(1-a)z) \in F^-(t,x)$ for $a \in [0,1]$. The sets $F^-(t,x)$ are of form 
\begin{equation}\label{zb_wypukle}
    F^-(t,x)=\{(y_0,y) \in \mathbb{R}^2: y_0 \leq \big(\alpha - \beta(Q^f+q^c)-k^c\big)q^ce^{-rt}, y=-q^c \text{ dla } q^c \in\mathbb{U} \}.
\end{equation}
If $(y_0,y)$ and $(z_0,z)$ belong to $F^-(t,x)$, then there exist $q^c_y, q^c_z \in \mathbb{U}$ such that 
 \begin{equation}\label{wklesle_1}
 \begin{cases}
    y_0 \leq \big(\alpha - \beta(Q^f+q^c_y)-k^c\big)q^c_ye^{-rt}& \\
    y=-q^c_y             &
\end{cases}
 \end{equation}
 \begin{equation}\label{wklesle_2}
 \begin{cases}
    z_0 \leq \big(\alpha - \beta(Q^f+q^c_z)-k^c\big)q^c_ze^{-rt}& \\
    z=-q^c_z             &
\end{cases}
\end{equation}
Multiplying the first system of equations and inequalities by $a$ and adding it to the second system multiplied by $1-a$ does not take us out of the set $F^-(t,x)$. This follows from the linearity of the second line in the systems given by \eqref{wklesle_1} and \eqref{wklesle_2} along with the concavity with respect to $q^c_y$ and $q^c_z$ of the functions on the right-hand side of the inequalities in the first line of the systems given by \eqref{wklesle_1} and \eqref{wklesle_2}. Therefore, $(ay_0+(1-a)z_0, ay+(1-a)z) \in F^-(t,x)$. Hence, we conclude that in the analyzed optimization problem an optimal control exists.
\end{proof}

The proof for the problem with a fixed terminal time \eqref{Bolza_funkcjonal}-\eqref{bolza_zmianastanu} is significantly analogous. This problem corresponds to the initial phase in ''share'' type of strategy, which is phase $S$.

\begin{proposition}\label{stw_zal_tw_o_istnieniu}
    The assumptions of Theorem \ref{tw_o_istnieniu} are satisfied for the problem with a fixed terminal time $T_1$ and a free terminal state, described by \eqref{Bolza_funkcjonal}-\eqref{bolza_zmianastanu}, assuming continuity of the final payoff and without switching between phases.
\end{proposition}

\begin{proof}
Consider the cartel's optimization problem with a fixed terminal time $T_1$, described by \eqref{Bolza_funkcjonal}-\eqref{bolza_zmianastanu}. Obviously, the target set $\mathbb{S}$ given by the formula
\[\mathbb{S}(t)=
\begin{cases}
    \mathbb{R} \quad \quad t=T_1& \\
    \emptyset \quad \quad \text{otherwise}.            &
\end{cases}
\]  
is attainable for the trajectory $S^c$, since the codomain of  this function is the real line $\mathbb{R}$. 

Now, we show that $\varphi^c$ satisfies $(\mathbb{\hat{H}})$. We need to find $C>0$ such that 
\begin{equation*}
    |-q^c(t)|\leq C\Big(1+\big|-\int_t^{T_1}q^c(t)\big|\Big).
\end{equation*}
Since the constraint $p(t)=\alpha-\beta q^c\geq 0$ implies $|q^c(t)|\leq \frac{\alpha}{\beta}$, it is sufficient to take $C=\frac{\alpha}{\beta}$.

When it comes to the continuity and differentiability conditions of functions $f^c$ and $\varphi^c$, is sufficient to show that the functions 
\begin{itemize}
        \item $f^c(t,S^c,q^c)=(\alpha-\beta(Q^f+q^c)-k^c)q^c$,
        \item $\varphi^c\big(t,S^c(t)\big)=-q^c(t)$
    \end{itemize}
are continuous with respect to all variables and are of $C^1$ class with regards to $S^c$. Note that neither of the functions $f^c, \varphi^c$ depends explicitly on $S^c$, while $f^c$ and $\varphi^c$ depend, respectively, quadratically and linearly on $q^c$.

The terminal payoff function $g^c$ is given by the formula
\begin{equation}\label{wyplata_konc_strategia2}
        g^c\big(T_1,S^c\big)= \int_{T_1}^{T^c}e^{-rt}\Big(\alpha-\beta\big(\overline{q}^c(t)+Q^f(t)\big)-k^c\Big)\overline{q}^c(t)dt,
    \end{equation}
where $\overline{q}^c$ is the optimal solution in the phase following $T_1$. Regarding the continuity of the function $g^c$ with respect to all variables and its belonging to the class $C^1$ with respect to $S^c$, it should be noted that since $T^c$ depends on $S^c(T_1)$, then $g^c$ does as well. Therefore, the function $g^c$ depends on the optimal control in the phase following $T_1$, assumption is necessary for $g^c$ to be continuous with respect to all its variables. However, in the main part of this work (in Subsection \ref{s_conclusion_arising}, we have demonstrated that $\overline{q}^c(t)$ is a continuous function within the interval $[T_1,T^c]$.

For a fixed end time $T_1$, the proof of the convexity of the sets  $F^-(t,x)$ is analogous.
\end{proof}

\subsubsection{Verification of the Maximum Principle assumptions in the context of the considered model}

In this subsection, we verify the assumptions of Theorems \ref{ZMP_zwiazane} and \ref{zmp_zwykla}. Since in Subsection \ref{podrozdzial_weryf_tw_o_istn} we checked the continuity and regularity of the relevant functions except for the final payment function given by equation \eqref{wyplata_konc_strategia2}, we need to verify the remaining assumptions in both cited theorems. For the application of Theorem \ref{ZMP_zwiazane} in the further part of the work, it is necessary to check that the gradient $\nabla\psi\big(T^*,S^c(T^*)\big)$ is nonzero. This is true because $\nabla\psi\left(T^*,S^c(T^*)\right)=1$. We do not conduct a similar reasoning for the fringe's problem since we optimize their problem statically without using the Maximum Principle.

We show that $\lambda_0^c \neq 0$ in Theorem \ref{ZMP_zwiazane} and the fact that $\gamma^c$ is constant.

\begin{proposition}\label{wniosek_stalosc_gammy}
    If $q^c$ is the optimal control for the cartel in the considered model, then in Theorem \ref{ZMP_zwiazane} corresponding to the final phase (thus, phase $C$ or $L$), we can assume $\lambda_0=1$. Additionally, in the same theorem, the function $\gamma^c$ is constant.
\end{proposition}

\begin{proof}

We start with proving the constancy of the function $\gamma^c$. Based on equation \eqref{zmp_zwiazane_1} in Theorem \ref{ZMP_zwiazane}, we obtain that $\dot{\gamma}^c=-\frac{\partial H^c}{\partial S^c}=0$, which means that $\gamma^c(t)=\gamma^c(T^c)$ is constant.

Now we prove that $\lambda_0^c \neq 0$. Assume the opposite, that is $\lambda_0^c=0$. If: \\
1. $\gamma^c=0$, then we obtain a contradiction, because according to the assumptions of Theorem \ref{ZMP_zwiazane}, the function $\gamma^c$ was supposed to be non-trivial. \\
2. $\gamma^c>0$, then the optimal control based on maximizing the Hamiltonian function $H^c$ would be $q^c(t) \equiv 0$ -- a contradiction, as zero control leads to zero net profits at every moment of the game, which certainly isn't an optimal strategy. \\
3. $\gamma^c<0$, then based on maximizing the cartel's Hamiltonian function, there is no finite optimal control, although earlier we showed its existence and that it is bounded -- a contradiction.

Therefore, $\lambda_0^c \neq 0$ and without loss of generality, assume $\lambda^c_0=1$. This is justified because it is sufficient to scale the maximization problem.
\end{proof}

\begin{proposition}
    For the initial phase (i.e., phase $S$), the function $\gamma^c$ in Theorem \ref{zmp_zwykla} is constant.
\end{proposition}

\begin{proof}
    Based on equation \eqref{zmp_2_wk} in Theorem \ref{zmp_zwykla}, we obtain $\dot{\gamma}^c=-\frac{\partial H^c}{\partial S^c}=0$, thus $\gamma^c(t)=\gamma^c(T_1)$ is constant.
\end{proof}

\section{Proofs}\label{s_proofs_appendix}

\subsection{Proof of Proposition \ref{stw_o_brzegowym_ster_opt}}\label{ss_proofs_stw_o_brzegowym_ster_opt}
First, we determine this interval. In this phase, only the cartel extracts its resource, so it sets the price in the range $p(t)\in [0,b]$. We determine the lower bound of the interval from the constraint $\alpha-\beta q^c\leq b$, which leads to $q^c\geq\frac{\alpha-b}{\beta}$. The upper bound of the interval is determined analogously, because $p(t)=\alpha-\beta q^c\geq 0$. It follows that $q^c(t)\leq \frac{\alpha}{\beta}$. Therefore, we assume that the set of cartel control parameters $q^c$ belongs to the interval
\begin{equation}\label{odcinek_zbior}
    q^c \in \Big[\frac{\alpha-b}{\beta},\frac{\alpha}{\beta}\Big].
\end{equation}

In a situation where
\begin{equation*}\label{dolny_kraniec_przedzialu}
    q^c(t)=\frac{\alpha-b}{\beta},
\end{equation*}
we say that the cartel's control function equals the lower bound of the interval. Analogously, we can define the situation in which the control function is equal to the upper bound of the interval. For it, the following condition is satisfied:
\begin{equation*}\label{gorny_kraniec_przedzialu}
q^c(t)=\frac{\alpha}{\beta}.
\end{equation*}
In the further part of the study, we justify that such a level of control cannot be optimal.

First, we demonstrate that the set of optimal control parameters in this phase does not belong to the interior of the interval $[\frac{\alpha-b}{\beta},\frac{\alpha}{\beta}]$, thus it belongs to the boundary of the interval, i.e., it equals either the lower or upper bound of the interval on a set of positive measure. Then, we show that in this phase, the optimal control of the cartel \textit{almost} always takes values equal to the lower bound of the interval.

\begin{lemma}\label{l_boundary values}
    The optimal control of the cartel in phase $L$ takes boundary values, i.e. the set of optimal control parameters does not lie within the interior of the interval $[\frac{\alpha-b}{\beta},\frac{\alpha}{\beta}]$.
\end{lemma}

\begin{proof}
We first assume that the optimal control of the cartel does not take boundary values for each $t$ in the described phase, which means that it lies within the interval $q^c \in (\frac{\alpha-b}{\beta},\frac{\alpha}{\beta})$. We check if such a situation is possible. Based on equation \eqref{max_zero} in Theorem \ref{ZMP_zwiazane}: $\max_{q^c}H^c(T^c,S^c_*(T^c),q^c,\gamma^c_*(T^c))=0$, meaning that for the optimal $q^c$
\begin{equation*}\label{ham_Zero}
    (\alpha-k^c-\beta q^c)q^ce^{-rT^c}-q^c\gamma^c=0,
\end{equation*}
which implies $\hat{q}^c(T^c)=0$ or $\hat{q}^c(T^c)=\frac{\alpha-k^c}{\beta}-\frac{\gamma^c}{\beta}e^{rT^c}$. The first case is impossible, because it does not satisfy the constraint $q^c(T^c) \in (\frac{\alpha-b}{\beta},\frac{\alpha}{\beta})$. Hence, $\hat{q}^c(T^c)=\frac{\alpha-k^c}{\beta}-\frac{\gamma^c}{\beta}e^{rT^c}$. 

Notice that by the necessary condition for maximizing the Hamiltonian function, given by formula \eqref{argmax_u} from Theorem \ref{ZMP_zwiazane}, we obtain that at points $q^c$, such that $q^c \in (\frac{\alpha-b}{\beta},\frac{\alpha}{\beta})$:
\begin{equation}\label{argmax_diff_niebrzeg_1}
    \frac{\partial H^c}{\partial q^c}=(\alpha-k^c-2\beta q^c)e^{-rt}-\gamma^c=0.
\end{equation}
We solve the equation \eqref{argmax_diff_niebrzeg_1} with respect to $q^c$ and we obtain
\begin{equation}\label{q^c_dowod_brzegowosci}
   q^c(t)=\frac{\alpha-k^c}{2\beta}-\frac{\gamma^c}{2\beta}e^{rt}  \quad \text{ for } t<T^c.
 \end{equation}

To complete the proof, we want to show that when taking the limit $t \rightarrow T^c$ there is a contradiction between the condition \eqref{argmax_u} from Theorem \ref{ZMP_zwiazane} and the condition \eqref{max_zero}. Notice that the control $q^c$ is almost everywhere given by the formula \eqref{argmax_u}, which can be expressed as
\begin{equation}\label{arg_max_sterow}
    q^c_{\max}(t):=\Argmax_{q^c \in [\frac{\alpha-b}{\beta},\frac{\alpha}{\beta}]}H^c=\min\Big(\max\big(\frac{\alpha-k^c}{2\beta}-\frac{\gamma^c}{2\beta}e^{rt},\frac{\alpha-b}{\beta}\big),\frac{\alpha}{\beta}\Big),
\end{equation}
which, for points within the interior of the considered interval, is identical to $q^c(t)$ given by equation \eqref{q^c_dowod_brzegowosci}. The function $q^c_{\max}(t)$ is continuous in $T^c$.

Define the set $\mathbb{T}$ as follows:
\begin{equation*}
\mathbb{T} \text{ -- the set of points where $q^c(t)$ is given by equation \eqref{arg_max_sterow}}.
\end{equation*}

Notice that $\mathbb{T}$ is a set of full measure. Without loss of generality, assume that $T^c \in \mathbb{T}$. Such an assumption is justified, because any potential discontinuity point of cartel's control (i.e. $T^c$) does not affect the value of the integral expressed by equation \eqref{eq_kartel_calka}, which represents the discounted profit.

Consider $\lim_{t_n \rightarrow T^c}q^c(t)$, where $(t_n)_{n \in \mathbb{N}}$ is a sequence of points that belong to the set $\mathbb{T}$. In other words, $(t_n)_{n \in \mathbb{N}}$ is a sequence of points for which $q^c(t)$ is expressed by formula \eqref{arg_max_sterow}. We obtain
\begin{equation*}
    \lim_{t_n \rightarrow T^c}q^c(t_n)=q_{\max}^c(T^c)=\frac{1}{2}\hat{q}^c(T^c)\neq \hat{q}^c(T^c),
\end{equation*}
This implies a contradiction between conditions \eqref{argmax_u} and \eqref{max_zero}. Therefore, the optimal control does not come from the interior of the interval $[\frac{\alpha-b}{\beta},\frac{\alpha}{\beta}]$. Thus, we have proven that the optimal control denoted as $q_*^c$ takes boundary values. 
\end{proof}

\begin{corollary}\label{wniosek_o_zbiorze_miary}
    The optimal control of the cartel takes boundary values on a set of measure greater than zero.
\end{corollary}

\begin{proof}
 The optimal control of the cartel is at the boundary on a set of measure greater than zero, because the control function $q^c(t)$ given by \eqref{q^c_dowod_brzegowosci} is continuous within the interval of its values and monotonic, since the function $\gamma^c$ is constant (see: Appendix \ref{appendix_tech_proofs}, Proposition \ref{wniosek_stalosc_gammy}). Therefore, at some point $\overline{t}$ such that $T^f\leq\overline{t}<T^c$ control parameter $q^c$ reaches the boundary. Consequently, $q^c_{\max}$ given by \eqref{arg_max_sterow} becomes constant after this moment.
\end{proof}

\begin{proof} \emph{of Proposition \ref{stw_o_brzegowym_ster_opt}}
    Consider the lower bound of the interval $[\frac{\alpha-b}{\beta},\frac{\alpha}{\beta}]$. Since the optimal control of the cartel neither reaches the upper bound nor is interior, by Theorem \ref{tw_o_istnieniu} the optimal control function of the cartel exists and is equal to the lower bound almost everywhere on the set $[\overline{t},T^c]$. This is a consequence of Corollary \ref{wniosek_o_zbiorze_miary}. Therefore, let the optimal control satisfy $\lim_{\tilde{t}_n \rightarrow T^{c-}} q^c(\tilde{t}_n)=q^c_{\max}(T^c)=\frac{\alpha-b}{\beta}$ for a sequence of points $\tilde{t}_1,\dots,\tilde{t}_n, T^c \in \mathbb{T}$ and $q^c_{\max}$ is given by \eqref{arg_max_sterow}. The control $q^c(t)$ is almost everywhere equal to $q^c_{\max}(t)=\frac{\alpha-b}{\beta}$ for $t$ in the set $[\overline{t},T^c]$. Thus, the control $q^c_*(t)=\frac{\alpha-b}{\beta}$ is optimal almost everywhere on the set $[\overline{t},T^c]$. The fact that equality holds on a nonempty but measure-zero set does not affect the value of the discounted net profit expressed through the integral \eqref{eq_kartel_calka}.
\end{proof}

\subsection{Proof of Proposition \ref{stw_o_ograniczeniu_wszedzie_brzegowa}}\label{ss_proofs_stw_o_ogr_wszedzie_brzegowa}
Since $\gamma^c=(b-k^c)e^{-rT^c}>0$ (as $\gamma^c$ is constant), the control computed from equation \eqref{argmax_diff_niebrzeg_1} of the form $q^c(t)=\frac{\alpha-k^c}{2\beta}-\frac{\gamma^c}{2\beta}e^{rt}$ is a decreasing function of $t$. Therefore, at $t=T^f$ this function reaches its maximum. Hence, for the cartel's control function to take boundary values everywhere, the condition that needs to be satisfied is:
\begin{equation*}
    \frac{\alpha-k^c}{2\beta}-\frac{\gamma^c}{2\beta}e^{rt} \leq \frac{\alpha-b}{\beta},
\end{equation*}
where equality holds in the case of $t=T^f$. After simple transformations, we obtain:
$\frac{\gamma^c}{2\beta}e^{rt} \geq \frac{2b-k^c-\alpha}{2\beta}$. Since $\gamma^c=(b-k^c)e^{-rT^c}$, we have $b\big(2-e^{r(t-T^c)}\big)+k^c(e^{r(t-T^c)}-1)-\alpha \leq 0$. Note that $b\big(2-e^{r(t-T^c)}\big)+k^c(e^{r(t-T^c)}-1)-\alpha \leq 2b-k^c-\alpha$, because we have assumed in Section \ref{s_formulation} that $k^c<b$. If we assume the condition $2b-k^c-\alpha\leq 0$, then the cartel's control function in the discussed phase will always take the lower boundary value. Consequently, the price is $b$.  

Now we calculate the optimal extraction time for the cartel (denoted by $\hat{T}^c$). We need to compute the integral:
\begin{equation*}
    \int_{T^f}^{\hat{T}^c}\frac{\alpha-b}{\beta}dt=S^c_0.
\end{equation*}
Hence, $\frac{\alpha-b}{\beta}t\Big|_{T^f}^{\hat{T}^c}=S_0^c$ and after simple transformations, we find $\hat{T}^c=T^f+\frac{S_0^c \cdot \beta}{\alpha-b}$.
\qed

\subsubsection{The conclusion arising from the form of optimal controls in the final phases}\label{s_conclusion_arising}

So far, in Proposition \ref{stw_zal_tw_o_istnieniu}, we have verified the validity of all the assumptions of Theorem \ref{zmp_zwykla} (maximum principle) and Theorem \ref{tw_o_istnieniu} (existence of optimal control) for the problem with a fixed terminal time $T_1$ and a free terminal state (described by equations \eqref{Bolza_funkcjonal}-\eqref{bolza_zmianastanu}). We should verify the continuity condition and belonging to $C^1$ class with respect to $S^c$ of the terminal payoff function $g^c$ given by equation \eqref{wyplata_konc_strategia2}. 

\begin{corollary}\label{wniosek_o_spelnieniu_zalozen_w_PMP}
    The assumptions of Theorem \ref{zmp_zwykla} (maximum principle) and Theorem \ref{tw_o_istnieniu} (existence of optimal control) are satisfied for the problem with a fixed terminal time $T_1$ and a free terminal state, described by \eqref{Bolza_funkcjonal}-\eqref{bolza_zmianastanu}, without switching between phases. 
\end{corollary}
\begin{proof}
For the sake of brevity, we denote by $g^c$ the payoff function in the final phase described in Section \ref{faza_koncowa_1}.
The function $g^c$ is of class $C^1$ with respect to $S^c$. We calculate the value of $S^c(T^f)$ by \eqref{zmiana_st_c} and the fact that at $S^c(T^c)=0$:
\begin{equation*}
    S^c(T^f)=\int_{T^f}^{T^c}\frac{\alpha-b}{\beta}dt,
\end{equation*}
where $q^c(t)_*=\frac{\alpha-b}{\beta}$ is the optimal control calculated in Section \ref{faza_koncowa_1} for the final phase after depletion of the fringe resource. Hence, $S^c(T^f)=\frac{\alpha-b}{\beta}(T^c-T^f)$, indicating that $T^c$ depends linearly on $S^c(T^f)$. Consequently, by the Fundamental Theorem of Calculus, the function $g^c$ is of class $C^1$ with respect to $S^c(T^f)$ and thus, with respect to $S^c$ over its domain. 

Now, we need to prove that $g^c$ is a continuous function with respect to all variables. We have shown that in the final phases, the optimal control computed in Section \ref{faza_koncowa_1} is almost everywhere continuous on the interval $[T^f,T^c]$. Hence, the terminal payoff function $g^c$ is continuous because it is a continuous integral with respect to $T^c$ and therefore with respect to $S^c(T^f)$ and in general, with respect to $S^c$.

Therefore, we have verified that the assumptions of Theorems \ref{zmp_zwykla} and \ref{tw_o_istnieniu} are satisfied.
\end{proof}

\subsection{Proof of Proposition \ref{stw_o_ograniczeniu_zasobow}}\label{ss_proofs_stw_o_ograniczeniu_zasobow}

We can assume that the cartel will never deplete its resources first, as long as it allows the fringe to exploit theirs. In other words, such a scenario will never occur as long as $T^f<T^c$, where $T^f$ is the depletion time of the competitors' resources. Now we want to impose certain constraints on the model parameters. First, we calculate the time $T^f$. It suffices to compute the integral:
\begin{equation*}
    \int_0^{T^f}M^f ds=S_0^f,
\end{equation*}
which yields
\begin{equation}\label{teef}
    T^f=\frac{S_0^f}{M^f}.
\end{equation}
We can estimate the depletion time of the cartel's resources from below using the auxiliary constant $T^c_{aux}$, where $T^c_{aux}$ is the depletion time for the maximum control from the constraint \eqref{scen3_przedzial_par_ster} up to time $T^f$ and then according to the optimal control after the depletion of resources of the fringe. Assume that $T^f<T^c_{aux}$ (we want to impose the strongest possible constraint on the model parameters). We calculate this constant. Based on the constraint \eqref{scen3_przedzial_par_ster}, we notice that $q^c(t) < \frac{\alpha-k^f-M^f\beta}{\beta}$ in the discussed phase. Moreover, based on the result obtained in Subsection \ref{faza_koncowa_1}, the optimal control $q^c_*(t)=\frac{\alpha-b}{\beta}$ when $t\geq T^f$. Therefore,
\begin{equation}\label{t_pom}
    \int_0^{T^c}q^c(s)ds<\int_0^{T^f}\frac{\alpha-k^f-M^f\beta}{\beta}ds+\int_{T^f}^{T_{aux}^c}\frac{\alpha-b}{\beta}ds=S_0^c,
\end{equation}
which implies the inequality $T^f < T^c_{aux} \leq T^c$. By calculating the integrals on the right-hand side of the inequality \eqref{t_pom} using substitution \eqref{teef}, we obtain 
\begin{equation*}
  T^c_{aux}= \frac{\beta}{\alpha-b}\Big(S_0^c+\frac{k^f+M^f\beta-b}{\beta}\cdot\frac{S_0^f}{M^f}\Big).  
\end{equation*}
Based on the inequality $\frac{S_0^f}{M^f}=T^f<T^c_{aux}$ , we obtain:
\begin{equation}\label{nier_na_tpom}
  \frac{S_0^f}{M^f} < \frac{S_0^c \beta}{\alpha-\beta M^f - k^f}.  
\end{equation}
By assuming the condition \eqref{nier_na_tpom} restricting $S^c_0$ and $S^f_0$, we ensure that the cartel will never deplete its resources before competitors.
\qed

\subsection{Proof of Proposition \ref{stw_o_gamma^c}}\label{ss_proofs_stw_o_gamma_c}
To find the expression for $\gamma^c$, we calculate the partial derivative with respect to $S^c_{T^f}$ of the final payment function given by \eqref{pi^c}. Therefore, we compute $\frac{\partial \pi^c}{\partial S^c_{T^f}}$, which corresponds to the terminal condition \eqref{war_kon_gamma} in Theorem \ref{zmp_zwykla}. As a result, we obtain
\begin{equation*}
    \frac{\partial \pi^c}{\partial S^c_{T^f}}=\frac{b-k^c}{r}\cdot\frac{\alpha-b}{\beta}\cdot\Big(\frac{-r\beta}{\alpha-b}\Big)\cdot\big(-e^{-rT^c}\big),
\end{equation*}
    which implies $\frac{\partial \pi^c}{\partial S^c_{T^f}}=(b-k^c)e^{-rT^c}$. Therefore, based on Theorem \ref{zmp_zwykla}, we have
    \begin{equation}\label{gamma_c}
\gamma^c(T^f) = (b - k^c)e^{-r(T^f + T^c)}.
\end{equation}\qed

\subsection{Proof of Proposition \ref{stw_o_nierownosci_z_dyskontem}}\label{ss_proofs_stw_o_nierownosci_z_dyskontem}
We want the optimal control in this phase, given by \eqref{opt_ster_trojka} to be greater than the lower bound given in \eqref{scen3_przedzial_par_ster}, i.e., 
    \begin{equation*}
    \frac{\alpha-k^c-M^f\beta}{2\beta}-\frac{\gamma^c}{2\beta}e^{rt}>\frac{\alpha-b-M^f\beta}{\beta}.
    \end{equation*}
We take the limit of the left side of the inequality as $t \rightarrow T^f$. As a result of straightforward transformations and after substituting $\gamma^c$ from \eqref{gamma_c}, we obtain
    \begin{equation*}
    -\frac{b-k^c}{2\beta}e^{-rT^c}>\frac{\alpha+k^c-2b-M^f\beta}{2\beta}.
\end{equation*}
Hence,
\begin{equation*}
    b\big(2-e^{-rT^c}\big)>\alpha-M^f\beta+k^c\big(1-e^{-rT^c}\big)
\end{equation*}
and in this case, the optimal control $q^c_*(t)$ for this phase is greater than $\frac{\alpha-b-M^f\beta}{\beta}$ for every $t<T^f$.
\qed

\subsection{Proof of Proposition \ref{stw_o_skonczonsci_optymalnego_momentu_wyczerpania}}\label{ss_proofs_stw_o_skonczonosci_optymalnego_momentu}
After expressing equation \eqref{tece} explicitly, by substituting $s$ from \eqref{es} and transferring the right-hand side of the equation to the left, we obtain:
    \begin{equation}\label{rownanie_z_tece}
T^c-\frac{S_0^c\cdot\beta}{\alpha-b}-\frac{\beta}{\alpha-b}\cdot\Big(\frac{\alpha+k^c+M^f\beta-2b}{2\beta}\cdot T^f+\frac{(b-k^c)}{2\beta r}\cdot\Big(e^{rT^f}-1\Big)e^{-r(T^c+T^f)}\Big)=0.
\end{equation}
Define the function on the left-hand side of equation \eqref{rownanie_z_tece} as $\tau^c(t)$. It is continuous with respect to $t$ for $t>T^f$. By separating the terms $\tau^c$ that contain $t$ from the rest, we get:
\begin{equation*}
\begin{aligned}
\tau^c(t) =
& \; t 
- \frac{\beta}{\alpha - b} \cdot \frac{b - k^c}{2 \beta r} 
  \cdot \Bigl(e^{-r T^f} - 1 \Bigr) e^{-r (t + T^f)} \\[1mm]
& - \frac{S_0^c \cdot \beta}{\alpha - b}
- \frac{\beta}{\alpha - b} \cdot \Biggl( \frac{\alpha + k^c + M^f \beta - 2b}{2 \beta} \cdot T^f \Biggr).
\end{aligned}
\end{equation*}

We show that $\lim_{t\rightarrow T^{f+}} \tau^c(t)<0$. We approach with $t$ towards $T^f$ from the right, because we have assumed $t>T^f$. After straightforward calculations, we obtain:

\begin{equation}\label{limes_dowod_o_istnieniu}
\lim_{{t\rightarrow T^{f+}}} \tau^c(t)=-\frac{S_0^c\cdot\beta}{\alpha-b}-\frac{\beta}{\alpha-b}\cdot\frac{(b-k^c)}{2\beta r}\cdot\Big(e^{-rT^f}-1\Big)\cdot e^{-2rT^f}+\frac{\alpha-k^c-M^f\beta}{2(\alpha-b)}\cdot T^f.
\end{equation}
After simplifying the first and third terms of the right-hand side of equation \eqref{limes_dowod_o_istnieniu} and substituting $T^f$ with the exhaustion time of the reserves by the fringe given by \eqref{teef}, we obtain 
\begin{equation*}
\begin{aligned}
\lim_{t \rightarrow T^{f+}} \tau^c(t) =
& \; \frac{1}{2(\alpha - b)} 
    \cdot \Biggl( \frac{\alpha - k^c - M^f \beta}{2(\alpha - b)} \cdot \frac{S_0^f}{M^f} - 2 S_0^c \beta \Biggr) \\[1mm]
& - \frac{\beta}{\alpha - b} \cdot \frac{b - k^c}{2 \beta r} 
    \cdot \Bigl(e^{-r T^f} - 1 \Bigr) \cdot e^{-2 r T^f}.
\end{aligned}
\end{equation*}

Notice that the second term in the above sum, denoted as $\tau^c_2$ and given by $\tau^c_2=-\frac{\beta}{\alpha-b}\cdot\frac{(b-k^c)}{2\beta r}\cdot(e^{-rT^f}-1)\cdot e^{-2rT^f}$ is positive, but very small. It is a product of exponential functions with the base $\exp$ and the argument of the form $\text{const}_1 \cdot T^f$, where $\text{const}_1$ is a negative constant. 

On the other hand, the first term, denoted as $\tau^c_1$ and given by $\tau^c_1=\frac{1}{2(\alpha-b)}\cdot(\frac{\alpha-k^c-M^f\beta}{2(\alpha-b)}\cdot\frac{S_0^f}{M^f}-2S^c_0\beta)$ is significantly less than zero. Notice that
\begin{equation*}
\frac{\alpha-k^c-M^f\beta}{2(\alpha-b)}\cdot\frac{S_0^f}{M^f}-2 S^c_0 \beta<0,
\end{equation*}
which implies
\begin{equation}\label{nierownosc_frac_S_0^f}
\frac{S_0^f}{M^f}<\frac{4(\alpha-b)\cdot S_0^c \beta}{\alpha-k^c-M^f\beta}.
\end{equation}
In Proposition \ref{stw_o_ograniczeniu_zasobow}, we have assumed that $\frac{S_0^f}{M} < \frac{S_0^c \beta}{\alpha-\beta M^f - k^f}$. If
\begin{equation}\label{nierownosc_znacznie_mniejsza}
\frac{S_0^c \beta}{\alpha- k^f-M^f\beta}\ll\frac{4(\alpha-b)\cdot S_0^c \beta}{\alpha-k^c-M^f\beta},
\end{equation}
meaning the left side of the inequality is much smaller than the right side, then the inequality \eqref{nierownosc_frac_S_0^f} will be satisfied. Indeed, transforming inequality \eqref{nierownosc_znacznie_mniejsza} yields
\begin{equation*}
\frac{4(\alpha-b)}{\alpha-k^c-M^f\beta}>\frac{1}{\alpha-k^f-M^f\beta}.
\end{equation*}
Further simplifications result in
\begin{equation*}
4(\alpha-b)(\alpha-k^f-M^f\beta)>\alpha-k^c-M^f\beta.
\end{equation*}
The term $\alpha^2$ on the left side of the above inequality is significantly larger than the remaining terms. Indeed, in Subsection \ref{faza_koncowa_1}, we assumed that it is much greater than $k^f$ and $k^c$ and at least twice as large as $b$. Therefore, inequality \eqref{nierownosc_znacznie_mniejsza} holds, and its left side is much smaller than its right side. Consequently, the term $\tau^c_1$ is significantly below zero. Thus, we conclude that
\begin{equation*}
\lim_{{t\rightarrow T^{f+}}} \tau^c(t)=\tau^c_1(T^f)+\tau^c_2(T^f)<0.
\end{equation*}
By the Darboux theorem for a continuous function $\tau^c$ defined on the closed interval $[T^f, \tilde{t}^c]$, where $\tilde{t}^c$ is any time such that $\tau^c(\tilde{t}^c)>0$ as $\tau^c$ an increasing function of $t$, there exists a point such that $\pi^c(T^c)=0$. This point is the optimal exhaustion moment for the cartel's resources because it satisfies the equation \eqref{rownanie_z_tece}.
\qed

\section{Numerical calculations for the graphical analysis section}\label{s_graphical_appendix}

We define the half-planes that divide the regions in which cartel's net profit $\pi^c_i$ is the highest from choosing a strategy $i=1,2,3$. 

\subsection{Situation where net profit obtained from ''deter'' strategy is the highest}\label{s_graphical_appendix_h}

As a result of solving the inequality $\pi^c_1(k^f,k^c)>\pi^c_2(k^f,k^c)$ with respect to $(k^f,k^c)$, we get that if:

1. $k^f \in (87.3,96.2)$, then
\begin{equation*}
\begin{aligned}
k^c \in \Biggl(
0, \; 
& \frac{
-3.35055 \cdot 10^{16} + 9.60294 \cdot 10^{14} k^f 
- 9.60294 \cdot 10^{14} e^{-\frac{145.925}{225.5 - k^f}} k^f
}{A}
\Biggr)
\end{aligned}
\end{equation*}

 \begin{equation*}
 \frac{-4.25851*10^{12} (k^f)^2 + 
 4.25851*10^{12} e^{-\frac{145.925}{225.5 - k^f}} (k^f)^2}{
A}\Big),
\end{equation*}
where
\begin{equation*}
    A=6.26389 \cdot 10^{14} - 9.60294 \cdot 10^{14} e^{-\frac{145.925}{225.5 - k^f}} 
- 4.25851 \cdot 10^{12} k^f 
+ 4.25851 \cdot 10^{12} e^{-\frac{145.925}{225.5 - k^f}} k^f,
\end{equation*}
and the upper limit of the above interval will be denoted as $h_1(k^f)$,

2. $k^f \in (96.2,100)$, then $k^c \in (0,60)$.

\subsection{Situation where net profit obtained from: the ''share'' strategy is the highest}

As a result of solving the inequality $\pi^c_2(k^f,k^c)>\pi^c_1(k^f,k^c)>\pi^c_3(k^f,k^c)$ with respect to $(k^f,k^c)$, we get that if:

1. $k^f \in (0,87.3)$, then $k^c \in (0,60)$,

2. $k^f \in (87.3,96.2)$, then 
    $k^c \in (h_1(k^f),60)$.


\bibliography{cartel_fringe_bibliography}

@article{wiszniewska2002static,
  title={Static and dynamic equilibria in games with continuum of players},
  author={Wiszniewska-Matyszkiel, Agnieszka},
  journal={Positivity},
  volume={6},
  pages={433--453},
  year={2002},
  publisher={Springer},
doi={10.1023/A:1021331508984}
}

@article{chenavaz2021dynamic,
  title={Dynamic pricing of new products in competitive markets: A mean-field game approach},
  author={Chenavaz, R{\'e}gis and Paraschiv, Corina and Turinici, Gabriel},
  journal={Dynamic Games and Applications},
  volume={11},
  pages={463--490},
  year={2021},
  publisher={Springer},
doi     = {10.1007/s13235-020-00369-6}
}

@article{JAAKKOLA2019101,
title = {Carbon taxation, OPEC and the end of oil},
journal = {Journal of Environmental Economics and Management},
volume = {94},
pages = {101-117},
year = {2019},
issn = {0095-0696},
doi = {https://doi.org/10.1016/j.jeem.2019.01.011},
url = {https://www.sciencedirect.com/science/article/pii/S0095069618304467},
author = {Niko Jaakkola}
}

@article{KHAN199713,
title = {On the Existence of Pure Strategy Equilibria in Games with a Continuum of Players},
journal = {Journal of Economic Theory},
volume = {76},
number = {1},
pages = {13-46},
year = {1997},
issn = {0022-0531},
doi = {https://doi.org/10.1006/jeth.1997.2292},
url = {https://www.sciencedirect.com/science/article/pii/S0022053197922924},
author = {M.Ali Khan and Kali P. Rath and Yeneng Sun}
}

@article{HART1974453,
title = {Formation of cartels in large markets},
journal = {Journal of Economic Theory},
volume = {7},
number = {4},
pages = {453-466},
year = {1974},
issn = {0022-0531},
doi = {https://doi.org/10.1016/0022-0531(74)90114-8},
url = {https://www.sciencedirect.com/science/article/pii/0022053174901148},
author = {Sergiu Hart}
}

@article{KALAI201859,
title = {Large strategic dynamic interactions},
journal = {Journal of Economic Theory},
volume = {178},
pages = {59-81},
year = {2018},
issn = {0022-0531},
doi = {https://doi.org/10.1016/j.jet.2018.08.001},
url = {https://www.sciencedirect.com/science/article/pii/S0022053118304411},
author = {Ehud Kalai and Eran Shmaya}
}

@article{SHITOVITZ1974458,
title = {On some problems arising in markets with some large traders and a continuum of small traders},
journal = {Journal of Economic Theory},
volume = {8},
number = {4},
pages = {458-470},
year = {1974},
issn = {0022-0531},
doi = {https://doi.org/10.1016/0022-0531(74)90021-0},
url = {https://www.sciencedirect.com/science/article/pii/0022053174900210},
author = {Benyamin Shitovitz}
}

@article{aumann1964markets,
  title={Markets with a continuum of traders},
  author={Aumann, Robert J},
  journal={Econometrica: Journal of the Econometric Society},
volume  = {32},
  number  = {1–2},
  pages={39--50},
  year={1964},
  publisher={JSTOR},
doi     = {10.2307/1913732}
}

@article{benchekroun2012price,
  title={On price taking behavior in a nonrenewable resource cartel--fringe game},
  author={Benchekroun, Hassan and Withagen, Cees},
  journal={Games and Economic Behavior},
  volume={76},
  number={2},
  pages={355--374},
  year={2012},
  publisher={Elsevier},
doi     = {10.1016/j.geb.2012.06.008}
}

@article{BENCHEKROUN20091867,
title = {On nonrenewable resource oligopolies: The asymmetric case},
journal = {Journal of Economic Dynamics and Control},
volume = {33},
number = {11},
pages = {1867-1879},
year = {2009},
issn = {0165-1889},
doi = {https://doi.org/10.1016/j.jedc.2009.03.008},
url = {https://www.sciencedirect.com/science/article/pii/S0165188909001110},
author = {Hassan Benchekroun and Alex Halsema and Cees Withagen}
}

@article{salant1976exhaustible,
  title={Exhaustible resources and industrial structure: A {Nash-Cournot} approach to the world oil market},
  author={Salant, Stephen W},
  journal={Journal of Political Economy},
  volume={84},
  number={5},
  pages={1079--1093},
  year={1976},
  publisher={The University of Chicago Press},
doi={10.1086/260497}
}

@incollection{Lewis1980,
  author       = {Lewis, T. R. and Schmalensee, Richard},
  title        = {Cartel and Oligopoly Pricing of Nonreplenishable Natural Resources},
  booktitle    = {Dynamic Optimization and Mathematical Economics},
  publisher    = {Springer},
  address      = {Boston, MA},
  pages        = {133--156},
  year         = {1980},
  doi          = {10.1007/978-1-4684-3572-6_8}
}

@article{fabbri2024competition,
  title={On competition for spatially distributed resources in networks},
  author={Fabbri, Giorgio and Faggian, Silvia and Freni, Giuseppe},
  journal={Theoretical Economics},
  volume={19},
  number={2},
  pages={743--781},
  year={2024},
  publisher={Wiley Online Library},
doi={10.3982/TE4328}
}

@article{loury1986,
 ISSN = {00206598, 14682354},
 URL = {http://www.jstor.org/stable/2526505},
 author = {Glenn C. Loury},
 journal = {International Economic Review},
 number = {2},
 pages = {285--301},
 publisher = {[Economics Department of the University of Pennsylvania, Wiley, Institute of Social and Economic Research, Osaka University]},
 title = {A Theory of 'Oil'Igopoly: {Cournot} Equilibrium in Exhaustible Resource Markets with Fixed Supplies},
 volume = {27},
 year = {1986},
doi     = {10.2307/2526505}
}

@article{BENCHEKROUN20191,
title = {An oligopoly-fringe non-renewable resource game in the presence of a renewable substitute},
journal = {Journal of Economic Dynamics and Control},
volume = {105},
pages = {1-20},
year = {2019},
issn = {0165-1889},
doi = {https://doi.org/10.1016/j.jedc.2019.05.014},
url = {https://www.sciencedirect.com/science/article/pii/S0165188918303233},
author = {Hassan Benchekroun and Gerard {van der Meijden} and Cees Withagen}
}

@article{BENCHEKROUN2020102384,
title = {OPEC, unconventional oil and climate change - On the importance of the order of extraction},
journal = {Journal of Environmental Economics and Management},
volume = {104},
pages = {102384},
year = {2020},
issn = {0095-0696},
doi = {https://doi.org/10.1016/j.jeem.2020.102384},
url = {https://www.sciencedirect.com/science/article/pii/S0095069620301078},
author = {Hassan Benchekroun and Gerard {van der Meijden} and Cees Withagen}
}

@article{gilbert1978dominant,
title = {Dominant Firm Pricing Policy in a Market for an Exhaustible Resource},
author = {Gilbert, Richard},
year = {1978},
journal = {Bell Journal of Economics},
volume = {9},
number = {2},
pages = {385-395},
doi = {10.2307/3003589}
}

@article{newbery1981oil,
  title={Oil prices, cartels, and the problem of dynamic inconsistency},
  author={Newbery, David MG},
  journal={The Economic Journal},
  volume={91},
  number={363},
  pages={617--646},
  year={1981},
  publisher={Oxford University Press Oxford, UK},
doi={10.2307/2232829}
}

@article{groot1992,
 ISSN = {00130133, 14680297},
 URL = {http://www.jstor.org/stable/2234802},
 author = {Fons Groot and Cees Withagen and Aart de Zeeuw},
 journal = {The Economic Journal},
 number = {415},
 pages = {1478--1484},
 publisher = {[Royal Economic Society, Wiley]},
 title = {Note on the Open-Loop von {Stackelberg} Equilibrium in the Cartel Versus Fringe Model},
 volume = {102},
 year = {1992},
doi     = {10.2307/2234802}
}

@article{GROOT2003287,
title = {Strong time-consistency in the cartel-versus-fringe model},
journal = {Journal of Economic Dynamics and Control},
volume = {28},
number = {2},
pages = {287-306},
year = {2003},
issn = {0165-1889},
doi = {https://doi.org/10.1016/S0165-1889(02)00154-9},
url = {https://www.sciencedirect.com/science/article/pii/S0165188902001549},
author = {Fons Groot and Cees Withagen and Aart {de Zeeuw}}
}

@article{groot2000open,
  title={Open-loop von Stackelberg equilibrium in the cartel-vs.-fringe model},
  author={Groot, Fons and Withagen, Cees and De Zeeuw, Aart},
  journal={Energy Economics},
  volume={22},
  number={2},
  pages={209--223},
  year={2000},
  publisher={Elsevier},
doi     = {10.1016/S0140-9883(99)00029-8}
}

@article{fischer2017balancing,
  title={Balancing the carbon budget for oil: The distributive effects of alternative policies},
  author={Fischer, Carolyn and Salant, Stephen W},
  journal={European Economic Review},
  volume={99},
  pages={191--215},
  year={2017},
  publisher={Elsevier},
doi={10.1016/j.euroecorev.2017.04.003}
}

@book{bressan2007introduction,
  title={Introduction to the Mathematical Theory of Control},
  author={Bressan, A. and Piccoli, B.},
  isbn={9781601330024},
  series={AIMS Series on Applied Mathematics},
address   = {Springfield, MO, USA},
volume={1},
  year={2007},
  publisher={American Institute of Mathematical Sciences},
isbn      = {978-1-60133-002-4}
}

@book{zabczyk2020mathematical,
  author={Zabczyk, Jerzy},
  title     = {Mathematical Control Theory: An Introduction},
  series    = {Systems \& Control: Foundations \& Applications},
  volume    = {61},
  publisher = {Birkhäuser / Springer},
  address   = {Cham, Switzerland},
  year      = {2020},
  doi       = {10.1007/978-3-030-44778-6}
}

@article{van2012there,
  title={Is there really a green paradox?},
  author={van der Ploeg, Frederick and Withagen, Cees},
  journal={Journal of Environmental Economics and Management},
  volume={64},
  number={3},
  pages={342--363},
  year={2012},
  publisher={Elsevier},
doi={10.1016/j.jeem.2012.08.002}
}

@article{schmeidler1973equilibrium,
  title={Equilibrium points of nonatomic games},
  author={Schmeidler, David},
  journal={Journal of statistical Physics},
  volume={7},
  pages={295--300},
  year={1973},
  publisher={Springer},
doi     = {10.1007/BF01014905}
}

@article{lasry2007mean,
  title={Mean field games},
  author={Lasry, Jean-Michel and Lions, Pierre-Louis},
  journal={Japanese journal of mathematics},
  volume={2},
  number={1},
  pages={229--260},
  year={2007},
  publisher={Springer},
doi     = {10.1007/s11537-007-0657-8}
}

@article{caruso2020regularization,
  title={Regularization and approximation methods in {Stackelberg} games and bilevel optimization},
  author={Caruso, Francesco and Lignola, M Beatrice and Morgan, Jacqueline},
  journal={Bilevel optimization: Advances and next challenges},
  pages={77--138},
  year={2020},
  publisher={Springer},
doi={10.1007/978-3-030-52119-6_4}
}

@article{carlier2019softening,
  title={Softening bilevel problems via two-scale {Gibbs} measures},
  author={Carlier, Guillaume and Mallozzi, Lina},
  journal={Set-Valued and Variational Analysis},
  pages={1--23},
  year={2019},
  publisher={Springer},
doi={10.1007/s11228-021-00605-0}
}

@article{wiszniewska2014open,
  title={Open and closed loop Nash equilibria in games with a continuum of players},
  author={Wiszniewska-Matyszkiel, Agnieszka},
  journal={Journal of Optimization Theory and Applications},
  volume={160},
  pages={280--301},
  year={2014},
  publisher={Springer},
doi     = {10.1007/s10957-013-0317-5}
}

@article{benchekroun2015harvests,
  title={Harvests' lifespan and North--South market share rivalry},
  author={Benchekroun, Hassan and Benchekroun, Saad},
  journal={International Review of Economics \& Finance},
  volume={37},
  pages={114--124},
  year={2015},
  publisher={Elsevier},
doi={https://doi.org/10.1016/j.iref.2014.11.017}
}

@book{yong2012stochastic,
  title={Stochastic controls: Hamiltonian systems and HJB equations},
  author={Yong, Jiongmin and Zhou, Xun Yu},
  volume={43},
  year={2012},
doi={https://doi.org/10.1007/978-1-4612-1466-3},
address   = {New York, NY},
  publisher={Springer Science & Business Media}
}

@article{van2012applications,
  title={Applications of dynamic games to global and transboundary environmental issues: a review of the literature},
  author={Van Long, Ngo},
  journal={Strategic Behavior and the Environment},
  volume={2},
  number={1},
  pages={1--59},
  year={2012},
  publisher={Now Publishers, Inc.},
doi={10.1561/102.00000009}
}

@article{long2011dynamic,
  title={Dynamic games in the economics of natural resources: a survey},
  author={Van Long, Ngo},
  journal={Dynamic Games and Applications},
  volume={1},
  pages={115--148},
  year={2011},
  publisher={Springer},
doi={10.1007/s13235-010-0003-2}
}

@incollection{long2016resource,
  title={Resource Economics},
  author={Van Long, N.},
  booktitle={Handbook of dynamic game theory},
  year={2016},
  doi={10.1007/978-3-319-27335-8\_15-1},
  address   = {Cham, Switzerland},
pages={1--29},
  editor={Ba{\c{s}}ar, Tamer and Zaccour, Georges},
  publisher={Springer International Publishing}
}

@article{feichtinger2024optimal,
  title={Optimal drilling efforts and industry structure},
  author={Feichtinger, Gustav and Lambertini, Luca and Leitmann, George and Wrzaczek, Stefan},
  journal={AIMS Environmental Science},
  volume={11},
  number={4},
  pages={610--627},
  year={2024},
doi={10.3934/environsci.2024030}
}


\end{appendices}

\end{document}